\documentclass[12pt]{amsart}
\usepackage{amssymb}
\usepackage{amsmath}
\usepackage{mathrsfs}
\usepackage{enumerate}
\usepackage{amsthm}
\usepackage{cancel}
\usepackage{graphicx}
\newtheorem{theorem}{Theorem}[section]
\newtheorem{lemma}[theorem]{Lemma}
\newtheorem{corollary}[theorem]{Corollary}

\theoremstyle{definition}

\newtheorem{definition}[theorem]{Definition}
\newtheorem{remark}[theorem]{Remark}
\newtheorem{example}[theorem]{Example}
\numberwithin{equation}{section}
\usepackage{color}
\usepackage{hyperref}
\hypersetup{
    colorlinks=true, 
    linkcolor=blue,  
}
\mathchardef\hyphen="2D
\begin{document}
\allowdisplaybreaks
\title{Optimality of empirical measures as quantizers}
\author{March T.~Boedihardjo}
\address{Department of Mathematics, Michigan State University, East Lansing, MI 48824}
\email{boedihar@msu.edu}
\begin{abstract}
A common way to discretize a probability measure is to use an empirical measure as a discrete approximation. But how far from being optimal is this approximation in the $p$-Wasserstein distance? In this paper, we study this question in two contexts: (1) optimality among all uniform quantizers and (2) optimality among all (non-uniform) quantizers. In the first context, for $p=1$, we provide a complete answer to this question up to a $\mathrm{polylog}(n)$ factor. From the probabilistic point of view, this resolves, up to a $\mathrm{polylog}(n)$ factor, the problem of characterizing the expected 1-Wasserstein distance between a probability measure and its empirical measure in terms of non-random quantities. We also obtain some partial results for $p>1$ in the first context and for $p\geq 1$ in the second context.
\end{abstract}
\maketitle
\section{Introduction}
Most numbers need to be rounded to a certain decimal place before being stored in a computer. The same goes for probability measures. To store a probability measure, unless it is absolutely continuous and the density is easy to describe, one often needs to discretize the probability measure. While rounding a given number to the nearest, say, one hundredth is a simple procedure, discretizing a given probability measure while achieving the optimal approximation error is a highly nontrivial task \cite{GL}. Thus, in practice, one often discretizes a probability measure $\mu$ by simply using an empirical measure $\mu_{n}=\frac{1}{n}\sum_{i=1}^{n}\delta_{X_{i}}$ as a discrete approximation of the given measure $\mu$, where $X_{1},\ldots,X_{n}$ are independent samples of $\mu$. Here, $\delta_{x}$ denotes the probability measure with an atom at $x$ of mass $1$. Surprisingly, in many cases, the empirical measure $\mu_{n}$ gives an approximation error that is optimal up to a constant factor. But there are also cases for which the empirical measure is far from being optimal.

In this paper, we study the optimality of using the empirical measure $\mu_{n}$ as a discrete approximation of the given measure $\mu$. Before we continue our discussion, we need some terminologies.

Let $(E,d_{E})$ be a separable metric space and $p\geq 1$. Denote by $\mathcal{P}_{p}(E)$ the set of all probability measures $\mu$ on $E$ such that $\int_{E}d_{E}(x,x_{0})^{p}\,d\mu(x)<\infty$ for some $x_{0}\in E$ (or equivalently, for all $x_{0}\in E$). For $\mu,\nu\in\mathcal{P}_{p}(E)$, the {\it $p$-Wasserstein distance} between $\mu$ and $\nu$ is defined by
\[W_{p}(\mu,\nu):=\inf_{\gamma}\left(\int_{E\times E}d_{E}(x,y)^{p}\,d\gamma(x,y)\right)^{\frac{1}{p}},\]
where the infimum is over all probability distributions $\gamma$ on $E\times E$ with marginal distributions $\mu$ and $\nu$.
\begin{definition}\label{def1}
Suppose that $E$ is a separable metric space, $p\geq 1$, $\mu\in\mathcal{P}_{p}(E)$ and $n\in\mathbb{N}$. Define the {\it optimal quantization error} (see \cite{GL})
\[e_{n,p}(\mu):=\inf_{\nu}W_{p}(\mu,\nu),\]
where the infimum is over all probability measures $\nu$ on $E$ supported on at most $n$ points in $E$. Define the {\it optimal uniform quantization error} (see \cite{Chevallier, Quattrocchi})
\[b_{n,p}(\mu):=\inf_{x_{1},\ldots,x_{n}\in E}W_{p}\left(\mu,\frac{1}{n}\sum_{i=1}^{n}\delta_{x_{i}}\right).\]
\end{definition}
\begin{remark}\label{compare}
We always have
\[e_{n,p}(\mu)\leq b_{n,p}(\mu)\leq\mathbb{E}[W_{p}(\mu_{n},\mu)],\]
where $\mu_{n}$ is the empirical measure of $\mu$ with sample size $n$.
\end{remark}
\subsection{Empirical measure vs optimal uniform quantizer}
In this subsection, we study the optimality of the empirical measure $\mu_{n}$ among all discrete approximations of $\mu$ of the form $\frac{1}{n}\sum_{i=1}^{n}\delta_{x_{i}}$ for some $x_{1},\ldots,x_{n}\in E$. More precisely, let $X_{1},\ldots,X_{n}$ be independent samples of a probability measure $\mu$ on $E$. The question we study in this subsection is: How far is the tuple $(X_{1},\ldots,X_{n})$ from attaining the following infimum:
\[b_{n,p}(\mu)=\inf_{x_{1},\ldots,x_{n}\in E}W_{p}\left(\mu,\frac{1}{n}\sum_{i=1}^{n}\delta_{x_{i}}\right)?\]
In other words, is optimal selection of $x_{1},\ldots,x_{n}\in E$ significantly better than random selection of $x_{1},\ldots,x_{n}$ in terms of minimizing the error $W_{p}(\mu,\frac{1}{n}\sum_{i=1}^{n}\delta_{x_{i}})$, or is optimal selection only slightly better than random selection?
Simply put, is $b_{n,p}(\mu)\ll\mathbb{E}[W_{p}(\mu_{n},\mu)]$, or is $b_{n,p}(\mu)\sim\mathbb{E}[W_{p}(\mu_{n},\mu)]$, where $\mu_{n}=\frac{1}{n}\sum_{i=1}^{n}\delta_{X_{i}}$?

For a large class of absolutely continuous measures, the answer to this question is well understood in the asymptotic sense. Indeed, when $d>2p$, for an absolutely continuous probability measure $\mu$ on $\mathbb{R}^{d}$ with sufficiently many moments, the decay rate of $b_{n,p}(\mu)$, as $n\to\infty$, obtained in \cite[Theorem 1.1]{Quattrocchi} is asymptotically equivalent to the decay rate of $\mathbb{E}[W_{p}(\mu_{n},\mu)]$ obtained in \cite[Theorem 2]{BB} and \cite[Theorem 2]{DSS} up to a constant factor that depends on $p,d$. More precisely, under some moment assumptions, if $d>2p$ and $f$ is the density of $\mu$ with respect to the Lebesgue measure $\lambda$ on $\mathbb{R}^{d}$, then $b_{n,p}(\mu)$ and $\mathbb{E}[W_{p}(\mu_{n},\mu)]$ are both, up to constant factors, asymptotically equivalent to
\[n^{-1/d}\left(\int_{\mathbb{R}^{d}}f^{1-\frac{p}{d}}\,d\lambda\right)^{1/p}.\]

The practical use of quantization of measures gives rise to the need of understanding the non-asymptotic behaviors of $b_{n,p}(\mu)$ and $\mathbb{E}[W_{p}(\mu_{n},\mu)]$. However, these two quantities remain mysterious in the non-asymptotic sense for both discrete and absolutely continuous $\mu$ in general. Despite the mystery of each of these two quantities, the question of whether $b_{n,p}(\mu)\ll\mathbb{E}[W_{p}(\mu_{n},\mu)]$ or $b_{n,p}(\mu)\sim\mathbb{E}[W_{p}(\mu_{n},\mu)]$ turns out to be much more tractable as we see in this paper.

In this subsection, for $p=1$ (i.e., the 1-Wasserstein distance), we identify an obstruction, namely, the ``steady decay obstruction" that prevents the empirical measure $\mu_{n}$ from being optimal among all uniform quantizers, i.e., $b_{n,1}(\mu)\ll\mathbb{E}[W_{1}(\mu_{n},\mu)]$. More importantly, we show that this is the only such obstruction up to a $\mathrm{polylog}(n)$ factor.

We now explain this obstruction. When the $n$ is doubled, the errors always get smaller: $b_{2n,1}(\mu)\leq b_{n,1}(\mu)$ and $\mathbb{E}[W_{1}(\mu_{2n},\mu)]\leq\mathbb{E}[W_{1}(\mu_{n},\mu)]$. The former follows from the fact that for $x_{1},\ldots,x_{n}\in E$, we can write $\frac{1}{n}\sum_{i=1}^{n}\delta_{x_{i}}=\frac{1}{2n}\sum_{i=1}^{n}(\delta_{x_{i}}+\delta_{x_{i}})$, i.e., repeat each $x_{i}$. The latter follows from Lemma \ref{basicstability} below. Even though $b_{n,1}(\mu)$ and $\mathbb{E}[W_{1}(\mu_{n},\mu)]$ both get smaller when $n$ is doubled, they can decay in very different ways.

On the one hand, there is no control on how much $b_{2n,1}(\mu)$ can be smaller than $b_{n,1}(\mu)$. For example, when $\mu$ is the uniform measure on a metric space $(E,d_{E})$ consisting of $2n$ points in which $d_{E}(x,y)=1$ whenever $x\neq y$, we have $b_{n,1}(\mu)=\frac{1}{2}$ and $b_{2n,1}(\mu)=0$.

On the other hand, $\mathbb{E}[W_{1}(\mu_{n},\mu)]$ always decays steadily as $n$ gets large. See Lemma \ref{basicstability} and Lemma \ref{w1lb} below. The second statement of Lemma \ref{basicstability}, for example, shows that we always have
\begin{equation}\label{w1decayrule}
\frac{1}{2}\cdot\mathbb{E}[W_{1}(\mu_{n},\mu)]\leq\mathbb{E}[W_{1}(\mu_{2n},\mu)]\leq\mathbb{E}[W_{1}(\mu_{n},\mu)].
\end{equation}
Thus, since $\mathbb{E}[W_{1}(\mu_{n},\mu)]\geq b_{n,1}(\mu)$, it follows that
\[\mathbb{E}[W_{1}(\mu_{2n},\mu)]\geq\frac{1}{2}\cdot\mathbb{E}[W_{1}(\mu_{n},\mu)]\geq\frac{1}{2}\cdot b_{n,1}(\mu).\]
So when $b_{n,1}(\mu)\gg b_{2n,1}(\mu)$, this forces $\mathbb{E}[W_{1}(\mu_{2n},\mu)]\gg b_{2n,1}(\mu)$ and thus prevents the empirical measure $\mu_{2n}$ from being optimal among all discrete approximations of $\mu$ of the form $\frac{1}{2n}\sum_{i=1}^{2n}\delta_{x_{i}}$ for some $x_{1},\ldots,x_{2n}\in E$.
\vspace{0.5cm}

\begin{minipage}{0.4\textwidth}
\begin{center}
\includegraphics[scale=.4]{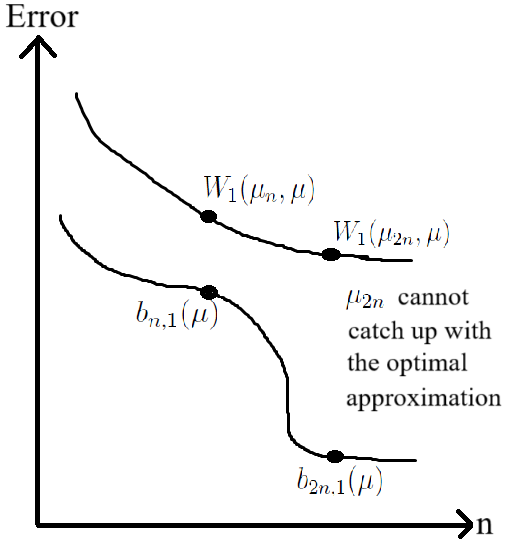}
\end{center}
\end{minipage}
\begin{minipage}{0.55\textwidth}
Simply put, since $\mathbb{E}[W_{1}(\mu_{n},\mu)]$ decays steadily, if a too steep drop $b_{n,1}(\mu)\gg b_{2n,1}(\mu)$ occurs at some $n$, then the empirical measure $\mu_{2n}$ cannot catch up with the optimal approximation. Moreover, even if the drop $b_{n,1}(\mu)\geq b_{2n,1}(\mu)$ is substantial but not too steep, multiple instances of such drops can accumulate and make $b_{n,1}(\mu)\ll\mathbb{E}[W_{1}(\mu_{n},\mu)]$ occur at some later $n$. This is the ``steady decay obstruction" that prevents an empirical measure from being optimal.
\end{minipage}
\vspace{0.5cm}

The first main result of this paper shows that up to a $\mathrm{polylog}(n)$ factor, the ``steady decay obstruction," as explained above, is the only obstruction that prevents the empirical measure $\mu_{n}$ from being optimal in the 1-Wasserstein distance.
\begin{theorem}\label{main1}
Suppose that $(E,d_{E})$ is a separable metric space, $\mu\in\mathcal{P}_{1}(E)$ and $n\in\mathbb{N}$. Then
\[\frac{1}{22\sqrt{\ln(2n)}}\cdot\max_{0\leq k\leq\lfloor\log_{2}n\rfloor}\sqrt{\frac{2^{k}}{n}}\cdot b_{2^{k},1}(\mu)\leq\mathbb{E}[W_{1}(\mu_{n},\mu)]\leq5\cdot\sum_{k=0}^{\lfloor\log_{2}n\rfloor}\sqrt{\frac{2^{k}}{n}}\cdot b_{2^{k},1}(\mu).\]
\end{theorem}
\begin{remark}
The upper and lower bounds in Theorem \ref{main1} differ by a factor of at most $110\sqrt{\ln(2n)}\cdot\lfloor\log_{2}n\rfloor$. From the probabilistic point of view, Theorem \ref{main1} resolves, up to a $\mathrm{polylog}(n)$ factor, the problem of characterizing $\mathbb{E}[W_{1}(\mu_{n},\mu)]$ in terms of non-random quantities. This problem has been studied since \cite[Section 3]{Dudley}. See \cite{BoissardW1, BL, DY, Talagrandmatching, WB}. Note that getting the correct log factor in estimating $\mathbb{E}[W_{1}(\mu_{n},\mu)]$ is a highly nontrivial task, e.g., when $\mu$ is the uniform measure on $[0,1]^{2}$. See \cite{Ajtai} and the survey paper \cite{Ledoux}.
\end{remark}
\begin{remark}
The characterization of $\mathbb{E}[W_{1}(\mu_{n},\mu)]$ in Theorem \ref{main1} has one drawback, namely, the notion of optimal uniform quantization error $b_{2^{k},1}(\mu)$ is not very well understood. To mitigate this, in the next subsection, we give a bound for $b_{2^{k},1}(\mu)$ in terms of the classical optimal quantization errors $e_{2^{j},q}(\mu)$, which have been well studied in the literature \cite{GL}. This enables us to directly bound $\mathbb{E}[W_{1}(\mu_{n},\mu)]$ in terms of the optimal quantization errors $e_{2^{j},q}(\mu)$. See Corollary \ref{main4}. We also illustrate the sharpness of this bound. See Corollary \ref{supmu} and Example \ref{abscont}.
\end{remark}
The second main result of this paper generalizes the upper bound in Theorem \ref{main1} to the $p$-Wasserstein distance.
\begin{theorem}\label{main2}
Suppose that $(E,d_{E})$ is a separable metric space, $p\geq 1$, $\mu\in\mathcal{P}_{p}(E)$ and $n\in\mathbb{N}$. Then
\[\left(\mathbb{E}[W_{p}(\mu_{n},\mu)^{p}]\right)^{1/p}\leq5\cdot\sum_{k=0}^{\lfloor\log_{2}n\rfloor}\left(\frac{2^{k}}{n}\right)^{1/(2p)}\cdot b_{2^{k},p}(\mu).\]
\end{theorem}
The following consequence of Theorem \ref{main2} shows that as a discrete approximation, the empirical measure $\mu_{n}$ gives an expected error $\mathbb{E}[W_{p}(\mu_{n},\mu)]$ that is, up to a $\log n$ factor, better than any sequence $\mu^{(n)}$ of uniform quantizers of $\mu$ with steadily decaying errors $W_{p}(\mu^{(n)},\mu)$.
\begin{corollary}
Let $(E,d_{E})$ be a separable metric space, $p\geq 1$ and $\mu\in\mathcal{P}_{p}(E)$. Let $J=\{2^{k}:k\in\mathbb{N}\cup\{0\}\}$. Suppose that $(\mu^{(n})_{n\in J}$ is a sequence of probability measures on $E$ such that each $\mu^{(n)}$ is of the form $\frac{1}{n}\sum_{i=1}^{n}\delta_{x_{i}}$ for some $x_{1},\ldots,x_{n}\in E$. If
\[W_{p}(\mu^{(2n)},\mu)\geq 2^{-1/(2p)}\cdot W_{p}(\mu^{(n)},\mu),\]
for all $n\in J$, then
\[W_{p}(\mu^{(n)},\mu)\geq\frac{1}{5(1+\log_{2}n)}\cdot \mathbb{E}[W_{p}(\mu_{n},\mu)],\]
for all $n\in J$.
\end{corollary}
\begin{proof}
Let $c_{n}=W_{p}(\mu^{(n)},\mu)$ for $n\in J$. Then $c_{2n}\geq 2^{-1/(2p)}c_{n}$ and $c_{n}\geq b_{n,p}(\mu)$ for all $n\in J$. Thus, the sequence $2^{k/(2p)}c_{2^{k}}$ is increasing in $k\in\mathbb{N}\cup\{0\}$, and so $2^{k/(2p)}c_{2^{k}}\leq n^{1/(2p)}c_{n}$ for all $n\in J$ and all $0\leq k\leq\log_{2}n$. By Theorem \ref{main2},
\begin{eqnarray*}
\mathbb{E}[W_{p}(\mu_{n},\mu)]&\leq&5\cdot\sum_{k=0}^{\log_{2}n}n^{-1/(2p)}\cdot 2^{k/(2p)} b_{2^{k},p}(\mu)\\&\leq&
5\cdot\sum_{k=0}^{\log_{2}n}n^{-1/(2p)}\cdot 2^{k/(2p)}c_{2^{k}}\\&\leq&5\cdot\sum_{k=0}^{\log_{2}n}n^{-1/(2p)}\cdot n^{1/(2p)}c_{n}=5(1+\log_{2}n)c_{n},
\end{eqnarray*}
for all $n\in J$.
\end{proof}
The following consequence of Theorem \ref{main2} gives a comparison between the decay rate of $\left(\mathbb{E}[W_{p}(\mu_{n},\mu)^{p}]\right)^{1/p}$ and the decay rate of $b_{n,p}(\mu)$. The proof uses a standard geometric series argument and is postponed to Section \ref{sectionremaining}.
\begin{corollary}\label{main2corollary}
Suppose that $(E,d_{E})$ is a separable metric space, $p\geq 1$ and $\mu\in\mathcal{P}_{p}(E)$ is not of the form $\delta_{x}$ (i.e., supported on more than one point). Then
\begin{align}\label{main2corollaryeq1}
\max\left(-\frac{1}{2},\,\limsup_{n\to\infty}\frac{\ln b_{n,p}(\mu)}{\ln n}\right)\leq&
\limsup_{n\to\infty}\frac{\ln\left(\mathbb{E}[W_{p}(\mu_{n},\mu)^{p}]\right)^{1/p}}{\ln n}\\\leq&
\max\left(-\frac{1}{2p},\,\limsup_{n\to\infty}\frac{\ln b_{n,p}(\mu)}{\ln n}\right).\nonumber
\end{align}
Moreover, for every $0\leq\beta<\frac{1}{2p}$, we have
\begin{equation}\label{main2corollaryeq2}
\limsup_{n\to\infty}n^{\beta}\cdot b_{n,p}(\mu)\leq\limsup_{n\to\infty}n^{\beta}\cdot\left(\mathbb{E}[W_{p}(\mu_{n},\mu)^{p}]\right)^{1/p}\leq\frac{10}{\frac{1}{2p}-\beta}\cdot\limsup_{n\to\infty}n^{\beta}\cdot b_{n,p}(\mu).
\end{equation}
\end{corollary}
\begin{remark}
The constant $-\frac{1}{2}$ in the lower bound in (\ref{main2corollaryeq1}) is optimal. Indeed, when $\mu$ is the uniform distribution on the unit interval $[0,1]$, we have $\left(\mathbb{E}[W_{p}(\mu_{n},\mu)^{p}]\right)^{1/p}\leq C\sqrt{\frac{p}{n}}$ for all $p\geq 1$ and $n\in\mathbb{N}$, where $C\geq 1$ is a absolute constant \cite[Theorem 4.8]{BobkovLedoux}. On the other hand, it is easy to see that $b_{n,p}(\mu)\leq\frac{1}{n}$. Therefore,
\[\limsup_{n\to\infty}\frac{\ln\left(\mathbb{E}[W_{p}(\mu_{n},\mu)^{p}]\right)^{1/p}}{\ln n}\leq-\frac{1}{2}\quad\text{and}\quad\limsup_{n\to\infty}\frac{\ln b_{n,p}(\mu)}{\ln n}\leq-1.\]
\end{remark}
\begin{remark}
The constant $-\frac{1}{2p}$ in the upper bound in (\ref{main2corollaryeq1}) is also optimal. Indeed, if $\mu$ is uniformly distributed on two points 0 and 1, we have $c\cdot n^{-1/(2p)}\leq\left(\mathbb{E}[W_{p}(\mu_{n},\mu)^{p}]\right)^{1/p}\leq C\cdot n^{-1/(2p)}$ for all $n\in\mathbb{N}$, where $C,c>0$ are absolute constants. On the other hand, it is easy to see that $b_{n,p}(\mu)\leq n^{-1/p}$ for all $n\in\mathbb{N}$. So
\[\limsup_{n\to\infty}\frac{\ln\left(\mathbb{E}[W_{p}(\mu_{n},\mu)^{p}]\right)^{1/p}}{\ln n}=-\frac{1}{2p}\quad\text{and}\quad\limsup_{n\to\infty}\frac{\ln b_{n,p}(\mu)}{\ln n}\leq-\frac{1}{p}.\]
\end{remark}
\begin{remark}
Some upper and lower bounds for $\limsup_{n\to\infty}\frac{\ln\mathbb{E}[W_{p}(\mu_{n},\mu)]}{\ln n}$ are obtained in \cite{WB}. However, it is not obvious how to compare those bounds with Corollary \ref{main2corollary}.
\end{remark}
\subsection{Empirical measure vs optimal quantizer}
In the previous subsection, we bound $\left(\mathbb{E}[W_{p}(\mu_{n},\mu)^{p}]\right)^{1/p}$ in terms of the optimal uniform quantization errors $b_{2^{k},p}(\mu)$. However, the notion of optimal uniform quantization error is less well understand than the classical notion of optimal quantization error $e_{n,p}(\mu)$ defined in Definition \ref{def1}.

In this subsection, we bound $\left(\mathbb{E}[W_{p}(\mu_{n},\mu)^{p}]\right)^{1/p}$ in terms of the optimal quantization errors $e_{r,q}(\mu)$. We have observed in the previous subsection that for $p=1$, up to a $\mathrm{polylog}(n)$ factor, the ``steady decay obstruction" is the only obstruction that prevents the empirical measure $\mu_{n}$ from being optimal among all the discrete approximations of the form $\frac{1}{n}\sum_{i=1}^{n}\delta_{x_{i}}$. However, in this subsection, when we compare the empirical measure $\mu_{n}$ with the larger class of discrete approximations of the form $\sum_{i=1}^{n}a_{i}\delta_{x_{i}}$ (i.e., non-uniform weights allowed), there are more obstructions that prevent the empirical measure $\mu_{n}$ from being optimal among this larger class of discrete approximations.

One obstruction is what we call the ``low likelihood obstruction." Let $\mu$ be a probability measure on a set $E$. Suppose that there is a subset $A\subset E$ for which $\mu(A)$ is small, i.e., $A$ has low likelihood of occurring. When one takes $n$ independent samples $X_{1},\ldots,X_{n}$ to form the empirical measure $\mu_{n}=\frac{1}{n}\sum_{i=1}^{n}\delta_{X_{i}}$, because $\mu(A)$ is small, not too many of those $X_{i}$ are in $A$. On the other hand, for the optimal quantization error $e_{n,p}(\mu)$, one is free to allocate as many or as little points to the set $A$. This extra freedom possessed by the optimal quantization could make its approximation error $e_{n,p}(\mu)$ significantly less than the approximation error $\mathbb{E}[W_{p}(\mu_{n},\mu)]$ of the empirical measure. To illustrate the ``low likelihood obstruction," we go through the following example.
\begin{example}\label{mixexample}
Consider the $\|\,\|_{\infty}$ metric on $\mathbb{R}^{3}$. Let $U$ be the uniform measure on the unit cube $[0,1]^{3}$. Let $0<\alpha<1$ be a small number. Consider the probability measure
\[\mu=\alpha U+(1-\alpha)\delta_{(0,0,4)}\]
on $\mathbb{R}^{3}$. This is a mixture of the uniform measure $U$ and the atomic measure $\delta_{(0,0,4)}$. Since $\alpha>0$ is small, the measure $\mu$ has most of its weights concentrated at the point $(0,0,4)$. Fix an integer $n\geq 2$ and let $m=\lfloor(n-1)^{1/3}\rfloor$. Let
\[S=\{(0,0,4)\}\cup\left\{\frac{1}{m},\frac{2}{m},\ldots,\frac{m}{m}\right\}^{3}.\]
Then $S$ contains $m^{3}+1\leq n$ points. So by Lemma \ref{enpalternative} below,
\begin{equation}\label{mixexampleeq1}
e_{n,1}(\mu)\leq\int_{\mathbb{R}^{3}}\mathrm{dist}(x,S)\,d\mu(x)=\alpha\int_{[0,1]^{3}}\mathrm{dist}(x,S)\,dU(x)\leq\frac{\alpha}{m}\leq C\cdot\alpha\cdot n^{-1/3},
\end{equation}
where $C\geq 1$ is a absolute constant and $\mathrm{dist}(x,S):=\inf_{y\in S}\|x-y\|_{\infty}$.

On the other hand, if we write the empirical measure $\mu_{n}=\frac{1}{n}\sum_{i=1}^{n}\delta_{X_{i}}$ where $X_{1},\ldots,X_{n}$ are independent samples of $\mu$, not too many of those samples are in $[0,1]^{3}$. Indeed, since $\mu([0,1]^{3})=\alpha$, we have $\mathbb{P}(X_{i}\in[0,1]^{3})=\alpha$ for each $i$, and so by Markov's inequality, there is an event with probability at least 1/2 in which at most $2\alpha\cdot n$ of the samples $X_{1},\ldots,X_{n}$ are in $[0,1]^{3}$. Observe that for every $x\in[0,1]^{3}$, we have
\[\mathrm{dist}(x,\{X_{1},\ldots,X_{n}\})\geq\mathrm{dist}(x,[0,1]^{3}\cap\{0,X_{1},\ldots,X_{n}\}),\]
since those $X_{i}$ that are not in $[0,1]^{3}$ must be equal to $(0,0,4)$ and so for those $X_{i}$, we have $\|x-X_{i}\|_{\infty}\geq\|x-0\|_{\infty}$ for all $x\in[0,1]^{3}$. Therefore,
\begin{eqnarray*}
W_{1}(\mu_{n},\mu)&\geq&\int_{\mathbb{R}^{3}}\mathrm{dist}(x,\{X_{1},\ldots,X_{n}\})\,d\mu(x)\\&\geq&
\alpha\int_{[0,1]^{3}}\mathrm{dist}(x,\{X_{1},\ldots,X_{n}\})\,dU(x)\\&\geq&
\alpha\int_{[0,1]^{3}}\mathrm{dist}(x,[0,1]^{3}\cap\{0,X_{1},\ldots,X_{n}\})\,dU(x)\geq
\alpha\cdot e_{N+1,1}(U),
\end{eqnarray*}
where $N=|[0,1]^{3}\cap\{X_{1},\ldots,X_{n}\}|$ is the number of $X_{i}$ that are in $[0,1]^{3}$, and the last inequality follows from Lemma \ref{enpalternative}. But one can show that the optimal quantization error $e_{N+1,1}(U)\geq c\cdot(N+1)^{-1/3}$ for some absolute constant $c>0$. (See the argument in comment (c) after Theorem 1 in \cite{FG}.) Therefore,
\[W_{1}(\mu_{n},\mu)\geq c\cdot\alpha\cdot(N+1)^{-1/3}.\]
However, as mentioned above, there is an event with probability at least 1/2 in which at most $2\alpha\cdot n$ of the $X_{i}$ are in $[0,1]^{3}$, i.e., $N\leq2\alpha\cdot n$. Thus, in this event, we have $W_{1}(\mu_{n},\mu)\geq c\cdot\alpha\cdot(2\alpha\cdot n+1)^{-1/3}\geq c'\cdot\alpha^{2/3}\cdot n^{-1/3}$, where $c'>0$ is another absolute constant, assuming that $\alpha\geq\frac{1}{n}$. Hence,
\begin{equation}\label{mixexampleeq2}
\mathbb{E}[W_{1}(\mu_{n},\mu)]\geq\frac{1}{2}c'\cdot\alpha^{2/3}\cdot n^{-1/3},
\end{equation}
if $\alpha\geq\frac{1}{n}$. So when $\alpha\ll 1$, this is significantly larger than the $e_{n,1}(\mu)$ in (\ref{mixexampleeq1}).
\end{example}
The above example is an illustration of the ``low likelihood obstruction." While the measure $\mu$ gives a small weight to the set $[0,1]^{3}$, this set requires more points to be allocated to in order for a quantizer to be accurate. But because the empirical measure $\mu_{n}$ fails to allocate sufficient number of points $X_{i}$ to $[0,1]^{3}$, it fails to be optimal.

Despite the suboptimality of $\mu_{n}$ in the 1-Wasserstein distance, the approximation error of $\mu_{n}$ in the 1-Wasserstein distance could still be as good as the approximation error of the optimal quantizer of $\mu$ in the 2-Wasserstein distance. Of course, this is not a fair comparison. But at least, if this is true, it shows that the empirical measure $\mu_{n}$ is not too much worse than the optimal quantizer. In the case of the above Example \ref{mixexample}, this is possible (in fact, true using Corollary \ref{main4} below). Indeed, if we adapt the argument in (\ref{mixexampleeq1}) for $e_{n,2}(\mu)$ instead of $e_{n,1}(\mu)$, we obtain
\begin{eqnarray*}
e_{n,2}(\mu)&\leq&\left(\int_{\mathbb{R}^{3}}\mathrm{dist}(x,S)^{2}\,d\mu(x)\right)^{1/2}\\&=&
\sqrt{\alpha}\left(\int_{[0,1]^{3}}\mathrm{dist}(x,S)\,dU(x)\right)^{1/2}\leq\frac{\sqrt{\alpha}}{m}\leq C\cdot\sqrt{\alpha}\cdot n^{-1/3}.
\end{eqnarray*}
This upper bound $C\cdot\sqrt{\alpha}\cdot n^{-1/3}$ is, up to a constant, greater than or equal to the lower bound $\frac{1}{2}c'\cdot\alpha^{2/3}\cdot n^{-1/3}$ in (\ref{mixexampleeq2}). So it is possible that $e_{n,2}(\mu)\geq c''\cdot\mathbb{E}[W_{1}(\mu_{n},\mu)]$ for some absolute constant $c''>0$.

In this subsection, we bound $\left(\mathbb{E}[W_{p}(\mu_{n},\mu)^{p}]\right)^{1/p}$ in terms of $e_{2^{k},q}(\mu)$ for $0\leq k\leq\lfloor\log_{2}n\rfloor$, where $q\in[p,\infty)$ is fixed.

The third main result of this paper gives an upper bound for the optimal uniform quantization error $b_{2^{k},p}(\mu)$ in terms of the optimal quantization errors $e_{2^{j},q}(\mu)$ for $0\leq j\leq k$, where $q\in[p,\infty)$ could depend on $j$. When combined with the second main result Theorem \ref{main2}, this enables us to directly bound $\mathbb{E}[W_{p}(\mu_{n},\mu)^{p}]$ in terms of the optimal quantization errors $e_{2^{k},q}(\mu)$ as shown in Corollary \ref{main4} below.
\begin{theorem}\label{main3}
Suppose that $(E,d_{E})$ is a separable metric space, $p\geq 1$, $\mu\in\mathcal{P}_{p}(E)$ and $k\in\mathbb{N}\cup\{0\}$. Then
\[b_{2^{k},p}(\mu)^{p}\leq 2\sum_{j=0}^{k}\inf_{q\in[p,\infty)}2^{(\frac{p}{q}-1)(k-j)}\cdot e_{2^{j},q}(\mu)^{p}.\]
\end{theorem}
\begin{corollary}\label{main4}
Suppose that $(E,d_{E})$ is a separable metric space, $p\geq 1$, $\mu\in\mathcal{P}_{p}(E)$ and $n\in\mathbb{N}$. Let $q\geq p$. Then
\[\left(\mathbb{E}[W_{p}(\mu_{n},\mu)^{p}]\right)^{1/p}\leq10\cdot\sum_{k=0}^{\lfloor\log_{2}n\rfloor}a_{k}\cdot e_{2^{k},q}(\mu),\]
where
\[a_{k}=\begin{cases}C(p,q)\cdot(2^{k}/n)^{1/(2p)},&q>2p\\(\lfloor\log_{2}n\rfloor+1)\cdot(2^{k}/n)^{1/(2p)},&q=2p\\C(p,q)\cdot(2^{k}/n)^{\frac{1}{p}-\frac{1}{q}},&q<2p\end{cases},\]
and $C(p,q)=2/\left|\frac{1}{2p}-\frac{1}{q}\right|$.
\end{corollary}
Unlike Theorem \ref{main1} and Corollary \ref{main2corollary}, we don't have a lower bound for $\left(\mathbb{E}[W_{p}(\mu_{n},\mu)^{p}]\right)^{1/p}$ that is somewhat comparable to the upper bound in Corollary \ref{main4}. Of course, we have $\left(\mathbb{E}[W_{p}(\mu_{n},\mu)^{p}]\right)^{1/p}\geq e_{n,p}(\mu)$, but Corollary \ref{main4} uses $e_{2^{k},q}(\mu)$ with $q\geq p$.

To illustrate the sharpness of Corollary \ref{main4}, we go through two uses of this result. These two uses demonstrate that merely knowing the quantization errors $e_{n,q}(\mu)$, for a suitable $q$, is often enough to get a fairly sharp bound for $\left(\mathbb{E}[W_{p}(\mu_{n},\mu)^{p}]\right)^{1/p}$.

For the first use of Corollary \ref{main4}, we take $q\to\infty$ to obtain Corollary \ref{supmu} below. It gives a worst case estimate of $\left(\mathbb{E}[W_{p}(\mu_{n},\mu)^{p}]\right)^{1/p}$ in terms of the geometry of the underlying metric space. The upper bound in Corollary \ref{supmu} can also be proved using \cite[Theorem 1.1]{BL}, except with a different constant.
\begin{corollary}\label{supmu}
Suppose that $(E,d_{E})$ is a compact metric space, $p\geq 1$ and $n\in\mathbb{N}$. For $m\in\mathbb{N}$, let
\[h(m)=\inf\{\epsilon>0:\,N(E,\epsilon)\leq m\},\]
where $N(E,\epsilon)$ is the $\epsilon$-covering number of $E$, i.e., the smallest possible cardinality of a set $S\subset E$ such that every $y\in E$ has distance at most $\epsilon$ from some point in $S$. Then
\begin{eqnarray*}
\frac{1}{4}\cdot\max_{0\leq k\leq\lfloor\log_{2}n\rfloor}\left(\frac{2^{k}}{n}\right)^{1/(2p)}h(2^{k})&\leq&\sup_{\mu\in\mathcal{P}_{p}(E)}\left(\mathbb{E}[W_{p}(\mu_{n},\mu)^{p}]\right)^{1/p}\\&\leq&
40p\cdot\sum_{k=0}^{\lfloor\log_{2}n\rfloor}\left(\frac{2^{k}}{n}\right)^{1/(2p)}h(2^{k}),
\end{eqnarray*}
where the supremum is over all probability measures $\mu\in\mathcal{P}_{p}(E)$.
\end{corollary}
\begin{remark}
The upper and lower bounds in Corollary \ref{supmu} differ by a factor of at most $160p\cdot\lfloor\log_{2}n\rfloor$.
\end{remark}
The following result is an asymptotic version of Corollary \ref{main4} in the regime $q<2p$.
\begin{corollary}\label{main4beta}
Suppose that $(E,d_{E})$ is a separable metric space, $p\geq 1$ and $\mu\in\mathcal{P}_{p}(E)$. Let $0\leq\beta<\frac{1}{2p}$, let $0<\epsilon<\frac{p^{2}}{2}(\frac{1}{2p}-\beta)$ and let $q=1/(\frac{1}{p+\epsilon}-\beta)$. Then
\[\limsup_{n\to\infty}n^{\beta}\cdot\left(\mathbb{E}[W_{p}(\mu_{n},\mu)^{p}]\right)^{1/p}\leq\frac{100p^{2}}{(\frac{1}{2p}-\beta)\epsilon}\cdot\limsup_{n\to\infty}n^{\beta}\cdot e_{n,q}(\mu).\]
\end{corollary}
\begin{example}\label{abscont}
Let $d\in\mathbb{N}$ and $p\geq 1$ be such that $d>2p$. Let $q=1/(\frac{1}{p+\epsilon}-\frac{1}{d})$ where $0<\epsilon<\frac{p^{2}}{2}(\frac{1}{2p}-\frac{1}{d})$. Let $\mu\in\mathcal{P}_{q+\epsilon}(\mathbb{R}^{d})$ and let $f$ be the density of the absolutely continuous part of $\mu$ with respect to the Lebesgue measure $\lambda$ on $\mathbb{R}^{d}$. By \cite[Theorem 6.2]{GL}, we have
\[\lim_{n\to\infty}n^{1/d}e_{n,q}(\mu)=C_{q,d}\left(\int_{\mathbb{R}^{d}}f^{\,d/(d+q)}\,d\lambda\right)^{(d+q)/(dq)}=C_{q,d}\left(\int_{\mathbb{R}^{d}}f^{1-\frac{p+\epsilon}{d}}\,d\lambda\right)^{1/(p+\epsilon)},\]
for some constant $C_{q,d}>0$ that depends only on $q$ and $d$. Applying Corollary \ref{main4beta} with $\beta=\frac{1}{d}$, we obtain
\[\limsup_{n\to\infty}n^{1/d}\cdot\left(\mathbb{E}[W_{p}(\mu_{n},\mu)^{p}]\right)^{1/p}\leq\frac{100p^{2}C_{q,d}}{(\frac{1}{2p}-\frac{1}{d})\epsilon}\left(\int_{\mathbb{R}^{d}}f^{1-\frac{p+\epsilon}{d}}\,d\lambda\right)^{1/(p+\epsilon)}.\]
Note that when $\epsilon\to 0$, the denominator goes to $0$, so the right hand side blows up. However, \cite[Theorem 2]{BB} and \cite[Theorem 2]{DSS} state that one can, in fact, bound
\[k_{p,d}\left(\int_{\mathbb{R}^{d}}f^{1-\frac{p}{d}}\,d\lambda\right)^{1/p}\leq\limsup_{n\to\infty}n^{1/d}\cdot\left(\mathbb{E}[W_{p}(\mu_{n},\mu)^{p}]\right)^{1/p}\leq K_{p,d}\left(\int_{\mathbb{R}^{d}}f^{1-\frac{p}{d}}\,d\lambda\right)^{1/p},\]
for some constants $K_{p,d},k_{p,d}>0$ that depend only on $p$ and $d$.
\end{example}
In addition to the above two uses of Corollary \ref{main4}, we note that Corollary \ref{main4}, when combined with the classical estimate of $e_{n,q}(\mu)$ for $\mu\in\mathcal{P}_{q+\epsilon}(\mu)$ \cite[Corollary 6.7]{GL}, yields a slightly weaker version of \cite[Theorem 1]{FG}, namely, the same statement but with two changes: (1) the $CM_{q}(\mu)^{p/q}$ is being replaced by $CM_{q+\epsilon}(\mu)^{p/(q+\epsilon)}$, where $\epsilon>0$ is arbitrarily small and $C$ can depend on $\epsilon$, and (2) when $d=2p$, the exponent of the $\log n$ factor is worse.

The rest of this paper is organized as follows. In Section 2, we explain the notation and state some basic results that are used throughout this paper. In Section 3, we prove Theorem \ref{main2}. In Section 4, we prove Theorem \ref{main1}. In Section 5, we prove Theorem \ref{main3}. In Section 6, we prove Corollary \ref{main2corollary}, Corollary \ref{main4}, Corollary \ref{supmu} and Corollary \ref{main4beta}.
\section{Notation and some basic results}
The set $\mathcal{P}_{p}(E)$ of all probability measures on $E$ with finite $p$th moments, the $p$-Wasserstein distance $W_{p}(\mu,\nu)$, the optimal quantization error $e_{n,p}(\mu)$ and the optimal uniform quantization error $b_{n,p}(\mu)$ are defined at the beginning of this paper.

If $a\geq 0$ then $\lfloor a\rfloor$ is the largest integer that is less than or equal to $a$.

Throughout this paper, if $W$ is a random variable, then
\[\|W\|_{L^{p}}:=(\mathbb{E}|W|^{p})^{\frac{1}{p}}.\]
For example, $\|W_{p}(\mu_{n},\mu)\|_{L^{p}}:=\left(\mathbb{E}[W_{p}(\mu_{n},\mu)^{p}]\right)^{1/p}$. To avoid confusion, the notation $\|\,\|_{L^{p}}$ is reserved exclusively for random variables and not for general functions.

If $\mu$ is a probability measure, its empirical measure from $n$ independent samples $X_{1},\ldots,X_{n}$ of $\mu$ is denoted by $\mu_{n}:=\frac{1}{n}\sum_{i=1}^{n}\delta_{X_{i}}$.

If $\mu$ is a probability measure on a set $E$ and $T:E\to E$ is measurable, then $T_{\#}\mu$ is the pushforward of $\mu$ by $T$, i.e., $T_{\#}\mu$ is the probability measure on $E$ defined by $(T_{\#}\mu)(A)=\mu(T^{-1}(A))$ for measurable $A\subset E$.

If $(E,d_{E})$ is a metric space and $S\subset E$, define $d_{E}(x,S):=\inf_{y\in S}d_{E}(x,y)$.

When $(E,d_{E})$ is a separable metric space and $\mu,\nu\in\mathcal{P}_{1}(E)$, by the Kantorovich-Rubinstein theorem, we have
\begin{equation}\label{krthm}
W_{1}(\mu,\nu)=\sup_{\|f\|_{\mathrm{Lip}}\leq 1}\left(\int_{E}f\,d\mu(x)-\int_{E}f\,d\nu(x)\right),
\end{equation}
where the supremum is over all $f:E\to\mathbb{R}$ with $\|f\|_{\mathrm{Lip}}\leq 1$. Here, the Lipschitz seminorm is defined by
\[\|f\|_{\mathrm{Lip}}:=\sup_{\substack{x,y\in E\\x\neq y}}\frac{f(x)-f(y)}{d_{E}(x,y)}=\sup_{\substack{x,y\in E\\x\neq y}}\frac{|f(x)-f(y)|}{d_{E}(x,y)}.\]
We now go through a list of basic results that are needed throughout this paper.
\begin{lemma}\label{emptrans}
Let $n\in\mathbb{N}$. Suppose that $\mu$ is a probability measure on $E$ and $T:E\to E$ is measurable. Then $T_{\#}\mu_{n}$ is an empirical measure of $T_{\#}\mu$.
\end{lemma}
\begin{proof}
Write $\mu_{n}=\frac{1}{n}\sum_{i=1}^{n}\delta_{X_{i}}$, where $X_{1},\ldots,X_{n}$ are independent samples of $\mu$. Then $T_{\#}\mu_{n}=\frac{1}{n}\sum_{i=1}^{n}\delta_{T(X_{i})}$ and $T(X_{1}),\ldots,T(X_{n})$ are independent samples of $T_{\#}\mu$. So $T_{\#}\mu_{n}$ is an empirical measure of $T_{\#}\mu$.
\end{proof}
\begin{lemma}[\cite{Villani}, page 5]\label{wppermute}
Let $(E,d_{E})$ be a metric space and $y_{1},\ldots,y_{r},z_{1},\ldots,z_{r}\in E$. Then for all $p\geq 1$,
\[W_{p}\left(\frac{1}{r}\sum_{i=1}^{r}\delta_{y_{i}},\frac{1}{r}\sum_{i=1}^{r}\delta_{z_{i}}\right)=\inf_{\sigma}\left(\sum_{i=1}^{r}d_{E}(y_{i},z_{\sigma(i)})^{p}\right)^{1/p},\]
where the infimum is over all permutations $\sigma:\{1,\ldots,r\}\to\{1,\ldots,r\}$.
\end{lemma}
\begin{lemma}\label{basicstability}
Let $(E,d_{E})$ be a separable metric space and $\mu\in\mathcal{P}_{1}(E)$. Then
\[\frac{m}{n}\cdot\mathbb{E}[W_{1}(\mu_{m},\mu)]\leq\mathbb{E}[W_{1}(\mu_{n},\mu)],\]
for all $m,n\in\mathbb{N}$ with $m\leq n$, and
\[\frac{1}{2}\cdot\mathbb{E}[W_{1}(\mu_{m},\mu)]\leq\mathbb{E}[W_{1}(\mu_{2m},\mu)]\leq\mathbb{E}[W_{1}(\mu_{m},\mu)],\]
for all $m\in\mathbb{N}$.
\end{lemma}
\begin{proof}
Let $X_{1},X_{2},\ldots$ be independent samples of $\mu$. We have
\[\mathbb{E}[W_{1}(\mu_{n},\mu)]=\mathbb{E}\sup_{\|f\|_{\mathrm{Lip}}\leq 1}\left(\frac{1}{n}\sum_{i=1}^{n}f(X_{i})-\int_{E}f\,d\mu\right).\]
Since the conditional expectation
\begin{align*}
&\mathbb{E}\left(\left.\sup_{\|f\|_{\mathrm{Lip}}\leq 1}\left(\frac{1}{n}\sum_{i=1}^{n}f(X_{i})-\int_{E}f\,d\mu\right)\right|\,X_{1},\ldots,X_{m}\right)\\\geq&
\sup_{\|f\|_{\mathrm{Lip}}\leq 1}\mathbb{E}\left(\left.\frac{1}{n}\sum_{i=1}^{n}f(X_{i})-\int_{E}f\,d\mu\right|\,X_{1},\ldots,X_{m}\right)\\=&
\sup_{\|f\|_{\mathrm{Lip}}\leq 1}\frac{1}{n}\sum_{i=1}^{m}\left(f(X_{i})-\int_{E}f\,d\mu\right),
\end{align*}
it follows that $\mathbb{E}[W_{1}(\mu_{n},\mu)]\geq\frac{m}{n}\cdot\mathbb{E}[W_{1}(\mu_{m},\mu)]$. This proves the first part of the result. To prove the second part, note that the lower bound $\frac{1}{2}\cdot\mathbb{E}[W_{1}(\mu_{m},\mu)]\leq\mathbb{E}[W_{1}(\mu_{2m},\mu)]$ follows from the first part by taking $n=2m$. For the upper bound, note that
\begin{align*}
&W_{1}(\mu_{2m},\mu)\\=&
\sup_{\|f\|_{\mathrm{Lip}}\leq 1}\left(\frac{1}{2m}\sum_{i=1}^{2m}f(X_{i})-\int_{E}f\,d\mu\right)\\\leq&
\frac{1}{2}\sup_{\|f\|_{\mathrm{Lip}}\leq 1}\left(\frac{1}{m}\sum_{i=1}^{m}f(X_{i})-\int_{E}f\,d\mu\right)+
\frac{1}{2}\sup_{\|f\|_{\mathrm{Lip}}\leq 1}\left(\frac{1}{m}\sum_{i=m+1}^{2m}f(X_{i})-\int_{E}f\,d\mu\right).
\end{align*}
Taking the expectation, we obtain $\mathbb{E}[W_{1}(\mu_{2m},\mu)]\leq\mathbb{E}[W_{1}(\mu_{m},\mu)]$.
\end{proof}
\begin{lemma}\label{mutmu}
Suppose that $(E,d_{E})$ is a separable metric space, $p\geq 1$, $\mu\in\mathcal{P}_{p}(E)$ and $T:E\to E$ is measurable. Then
\[W_{p}(\mu,T_{\#}\mu)\leq\left(\int_{E}d_{E}(x,T(x))^{p}\,d\mu(x)\right)^{1/p}.\]
\end{lemma}
\begin{proof}
Let $\gamma$ be the pushforward of $\mu$ by the map $x\mapsto(x,T(x))$. Then $\gamma$ is a probability measure on $E\times E$ with marginal distributions $\mu$ and $T_{\#}\mu$. So
\[W_{p}(\mu,T_{\#}\mu)\leq\left(\int_{E\times E}d_{E}(x,y)^{p}\,d\gamma(x,y)\right)^{1/p}=\left(\int_{E}d_{E}(x,T(x))^{p}\,d\mu(x)\right)^{1/p}.\]
\end{proof}
\begin{lemma}\label{mixlemma}
Let $(E,d_{E})$ be a separable metric space, $p\geq 1$ and $k\in\mathbb{N}$. For $1\leq j\leq k$, let $\mu^{(j)},\nu^{(j)}\in\mathcal{P}(E)$, $t_{j}\geq 0$ be such that $\sum_{j=1}^{k}t_{j}=1$. Then
\[W_{p}\left(\sum_{j=1}^{k}t_{j}\mu^{(j)},\sum_{j=1}^{k}t_{j}\nu^{(j)}\right)^{p}\leq\sum_{j=1}^{k}t_{j}\cdot W_{p}(\mu^{(j)},\nu^{(j)})^{p}.\]
\end{lemma}
\begin{proof}
For each $1\leq j\leq k$, let $\gamma^{(j)}$ be a probability measure on $E\times E$ with marginal distributions $\mu^{(j)}$ and $\nu^{(j)}$. Then the probability measure $\sum_{j=1}^{k}t_{j}\gamma^{(j)}$ on $E\times E$ has marginal distributions $\sum_{j=1}^{k}t_{j}\mu^{(j)}$ and $\sum_{j=1}^{k}t_{j}\nu^{(j)}$. So
\begin{eqnarray*}
W_{p}\left(\sum_{j=1}^{k}t_{j}\mu^{(j)},\sum_{j=1}^{k}t_{j}\nu^{(j)}\right)^{p}&\leq&
\int_{E}d_{E}(x,y)^{p}\;d\left(\sum_{j=1}^{k}t_{j}\gamma^{(j)}\right)(x,y)\\&=&
\sum_{j=1}^{k}t_{j}\int_{E}d_{E}(x,y)^{p}\,d\gamma^{(j)}(x,y).
\end{eqnarray*}
Taking infimum over all the $\gamma^{(j)}$, for $1\leq j\leq k$, completes the proof.
\end{proof}
\begin{lemma}\label{2samples}
Let $(E,d_{E})$ be a separable metric space, $p\geq 1$, $\mu\in\mathcal{P}_{p}(E)$ and $n\in\mathbb{N}$. Let $\mu_{n}$ and $\mu_{n}'$ be two independent empirical measures of $\mu$ with sample size $n$. Then
\[\|W_{p}(\mu_{n},\mu)\|_{L^{p}}\leq\|W_{p}(\mu_{n},\mu_{n}')\|_{L^{p}}\leq 2\cdot\|W_{p}(\mu_{n},\mu)\|_{L^{p}}.\]
\end{lemma}
\begin{proof}
The upper bound follows from $W_{p}(\mu_{n},\mu_{n}')\leq W_{p}(\mu_{n},\mu)+W_{p}(\mu_{n}',\mu)$. To prove the lower bound, without loss of generality, assume that $E$ is countable so that there is no measurability issue in what follows. Fix $\epsilon>0$. There exists a (random) probability measure $\gamma$ on $E\times E$ with marginal distributions $\mu_{n}$ and $\mu_{n}'$ such that
\[\int_{E\times E}d_{E}(x,y)^{p}\,d\gamma(x,y)\leq W_{p}(\mu_{n},\mu_{n}')^{p}+\epsilon.\]
Consider the probability measure $\eta$ on $E\times E$ defined by the conditional expectation
\[\eta(A)=\mathbb{E}(\gamma(A)|\,\mu_{n}),\]
for all measurable $A\subset E\times E$. It is easy to check that the marginal distributions of $\eta$ are $\mu_{n}$ and $\mathbb{E}\mu_{n}'=\mu$. Thus,
\begin{eqnarray*}
W_{p}(\mu_{n},\mu)^{p}&\leq&\int_{E\times E}d_{E}(x,y)^{p}\,d\eta(x,y)\\&=&
\mathbb{E}\left(\left.\int_{E\times E}d_{E}(x,y)^{p}\,d\gamma(x,y)\right|\,\mu_{n}\right)\leq\mathbb{E}(\left.W_{p}(\mu_{n},\mu_{n}')^{p}\right|\mu_{n})+\epsilon.
\end{eqnarray*}
Taking the full expectation on both sides and taking $\epsilon\to 0$, we obtain $\mathbb{E}[W_{p}(\mu_{n},\mu)^{p}]\leq\mathbb{E}[W_{p}(\mu_{n},\mu_{n}')^{p}]$. This completes the proof.
\end{proof}
\begin{lemma}\label{easybound}
Suppose that $(E,d_{E})$ is a separable metric space, $p\geq 1$, $\mu,\widetilde{\mu}\in\mathcal{P}_{p}(E)$ and $n\in\mathbb{N}$. Let $\widetilde{\mu}_{n}$ be an empirical measure of $\widetilde{\mu}$ with sample size $n$. Then
\[\bigg|\,\|W_{p}(\mu_{n},\mu)\|_{L^{p}}-\|W_{p}(\widetilde{\mu}_{n},\widetilde{\mu})\|_{L^{p}}\bigg|\leq 2W_{p}(\mu,\widetilde{\mu}).\]
\end{lemma}
\begin{proof}
Let $\gamma$ be a probability measure on $E\times E$ with marginal distributions $\mu$ and $\widetilde{\mu}$. Take independent samples $(X_{1},Y_{1}),\ldots,(X_{n},Y_{n})$ of $\gamma$. Since each of the random variables $X_{1},\ldots,X_{n}$ has distribution $\mu$ and since each of the random variables $Y_{1},\ldots,Y_{n}$ has distribution $\widetilde{\mu}$, we have $\mathbb{E}[W_{p}(\mu_{n},\mu)^{p}]=\mathbb{E}[W_{p}(\frac{1}{n}\sum_{i=1}^{n}\delta_{X_{i}},\mu)^{p}]$ and $\mathbb{E}[W_{p}(\widetilde{\mu}_{n},\widetilde{\mu})^{p}]=\mathbb{E}[W_{p}(\frac{1}{n}\sum_{i=1}^{n}\delta_{Y_{i}},\widetilde{\mu})^{p}]$. So
\begin{eqnarray*}
\|W_{p}(\mu_{n},\mu)\|_{L^{p}}-\|W_{p}(\widetilde{\mu}_{n},\widetilde{\mu})\|_{L^{p}}&=&\left\|W_{p}\left(\frac{1}{n}\sum_{i=1}^{n}\delta_{X_{i}},\mu\right)\right\|_{L^{p}}-\left\|W_{p}\left(\frac{1}{n}\sum_{i=1}^{n}\delta_{Y_{i}},\widetilde{\mu}\right)\right\|_{L^{p}}\\&\leq&
\left\|W_{p}\left(\frac{1}{n}\sum_{i=1}^{n}\delta_{X_{i}},\mu\right)-W_{p}\left(\frac{1}{n}\sum_{i=1}^{n}\delta_{Y_{i}},\widetilde{\mu}\right)\right\|_{L^{p}}
\\&\leq&
W_{p}(\mu,\widetilde{\mu})+\left\|W_{p}\left(\frac{1}{n}\sum_{i=1}^{n}\delta_{X_{i}},\frac{1}{n}\sum_{i=1}^{n}\delta_{Y_{i}}\right)\right\|_{L^{p}}\\&\leq&
W_{p}(\mu,\widetilde{\mu})+\left\|\left(\frac{1}{n}\sum_{i=1}^{n}d_{E}(X_{i},Y_{i})^{p}\right)^{\frac{1}{p}}\right\|_{L^{p}}\\&=&
W_{p}(\mu,\widetilde{\mu})+\left(\frac{1}{n}\sum_{i=1}^{n}\mathbb{E}[d_{E}(X_{i},Y_{i})^{p}]\right)^{\frac{1}{p}}\\&=&
W_{p}(\mu,\widetilde{\mu})+\left(\int_{E\times E}d_{E}(x,y)^{p}\,d\gamma(x,y)\right)^{\frac{1}{p}}.
\end{eqnarray*}
Taking infimum over $\gamma$, we obtain $\|W_{p}(\mu_{n},\mu)\|_{L^{p}}-\|W_{p}(\widetilde{\mu}_{n},\widetilde{\mu})\|_{L^{p}}\leq 2W_{p}(\mu,\widetilde{\mu})$. Interchanging the roles of $\mu$ and $\widetilde{\mu}$, we also obtain $\|W_{p}(\widetilde{\mu}_{n},\widetilde{\mu})\|_{L^{p}}-\|W_{p}(\mu_{n},\mu)\|_{L^{p}}\leq 2W_{p}(\mu,\widetilde{\mu})$.
\end{proof}
\begin{lemma}\label{measselect}
Let $(E,d_{E})$ be a separable metric space and let $S\subset E$ be finite. Then there is a measurable map $T:E\to S$ such that $d_{E}(x,T(x))=d_{E}(x,S)$ for all $x\in E$.
\end{lemma}
\begin{proof}
Enumerate the set $S=\{x_{1},\ldots,x_{r}\}$. For each $1\leq i\leq r$, let
\[A_{i}=\{x\in E:\,d_{E}(x,x_{i})=d_{E}(x,S)\}.\]
Then $A_{1}\cup\ldots\cup A_{r}=E$. Define $B_{1}=A_{1}$ and for $1\leq i\leq r$, define $B_{i}=A_{i}\backslash(A_{1}\cup\ldots\cup A_{i-1})$. Then $B_{1},\ldots,B_{r}$ forms a partition of $E$. For each $1\leq i\leq r$ and each $x\in B_{i}$, set $T(x)=x_{i}$. Then $T$ is measurable, and for each $1\leq i\leq r$ and each $x\in B_{i}$, we have $d_{E}(x,T(x))=d_{E}(x,x_{i})=d_{E}(x,S)$.
\end{proof}
\begin{lemma}\label{enpalternative}
Let $(E,d_{E})$ be a separable metric space, $n\in\mathbb{N}$, $p\geq 1$ and $\mu\in\mathcal{P}_{p}(E)$. Then
\[e_{n,p}(\mu)=\inf_{\substack{S\subset E\\|S|\leq n}}\left(\int_{E}d_{E}(x,S)^{p}\,d\mu(x)\right)^{1/p},\]
where the infimum is over all subsets $S\subset E$ containing at most $n$ points.
\end{lemma}
\begin{proof}
For every probability measure $\nu$ on $E$ supported on at most $n$ points in $E$, letting $\mathrm{supp}(\nu)$ be its support, we have
\[W_{p}(\mu,\nu)\geq
\inf_{\gamma}\left(\int_{E\times E}d_{E}(x,\mathrm{supp}(\nu))^{p}\,d\gamma(x,y)\right)^{1/p}=\left(\int_{E}d_{E}(x,\mathrm{supp}(\nu))^{p}\,d\mu(x)\right)^{1/p},\]
where the infimum is over all probability distributions $\gamma$ on $E\times E$ with marginal distributions $\mu$ and $\nu$. Thus,
\[e_{n,p}(\mu)\geq\inf_{\substack{S\subset E\\|S|\leq n}}\left(\int_{E}d_{E}(x,S)^{p}\,d\mu(x)\right)^{p}.\]
To prove the opposite inequality, let $S\subset E$ with $|S|\leq n$. By Lemma \ref{measselect}, there is a measurable map $T:E\to S$ such that $d_{E}(x,T(x))=d_{E}(x,S)$ for all $x\in S$. We have
\[\left(\int_{E}d_{E}(x,S)^{p}\,d\mu(x)\right)^{1/p}=
\left(\int_{E}d_{E}(x,T(x))^{p}\,d\mu(x)\right)^{1/p}\geq W_{p}(\mu,T_{\#}\mu)\geq e_{n,p}(\mu),\]
where the second step follows from Lemma \ref{mutmu} and the last step follows from the definition of $e_{n,p}(\mu)$.
\end{proof}
\section{Proof of the second main result}
In this section, we prove the second main result Theorem \ref{main2}. The idea of the proof here is quite similar to that in \cite[Theorem 1.1]{BL}, except for one key difference, namely, instead of controlling the diameter of each block in a partition, we control the number of blocks in a partition.

We begin with some terminologies. Let $(E,d_{E})$ be a finite metric space. If $\zeta$ is a measure on $E$ and $a\geq 0$, then $a\zeta$ is the measure on $E$ defined by
\[(a\zeta)(A)=a\zeta(A),\]
for $A\subset E$. If moreover, $T:E\to E$, then $T_{\#}\zeta$ is the measure on $E$ defined by
\[(T_{\#}\zeta)(A)=\zeta(T^{-1}(A)),\]
for $A\subset E$. Observe that the measures $T_{\#}(a\zeta)=aT_{\#}\zeta$ coincide for all $a\geq 0$.

For two measures $\zeta$ and $\eta$ on $E$, we write $\zeta\leq\eta$ if
\[\zeta(A)\leq\eta(A),\]
for all $A\subset E$. If $\zeta\leq\eta$ then $\eta-\zeta$ is the measure on $E$ defined by
\[(\eta-\zeta)(A)=\eta(A)-\zeta(A),\]
for $A\subset E$. Observe that if $\zeta\leq\eta$ and $T:E\to E$, then $T_{\#}\zeta\leq T_{\#}\eta$ and
\[T_{\#}(\eta-\zeta)=T_{\#}\eta-T_{\#}\zeta.\]

For two probability measures $\mu$ and $\nu$ on $E$, define the measure $\mu\wedge\nu$ on $E$ by
\[(\mu\wedge\nu)(\{x\})=\min(\mu(\{x\}),\nu(\{x\})),\]
for $x\in E$. Observe that $\mu\wedge\nu\leq\mu$ and $\mu\wedge\nu\leq\nu$ and the measures
\[\frac{\mu-(\mu\wedge\nu)}{1-(\mu\wedge\nu)(E)}\quad\text{and}\quad\frac{\nu-(\mu\wedge\nu)}{1-(\mu\wedge\nu)(E)}\]
on $E$ are probability measures, if $(\mu\wedge\nu)(E)<1$.
\begin{definition}\label{dollardef}
Let $\mu$ and $\nu$ be probability measures on a finite metric space $(E,d_{E})$. Let $p\geq 1$. Define
\[W_{p}^{\$}(\mu,\nu):=[1-(\mu\wedge\nu)(E)]^{\frac{1}{p}}W_{p}\left(\frac{\mu-(\mu\wedge\nu)}{1-(\mu\wedge\nu)(E)},\frac{\nu-(\mu\wedge\nu)}{1-(\mu\wedge\nu)(E)}\right).\]
When $\mu=\nu$, we have $(\mu\wedge\nu)(E)=1$ and we set $W_{p}^{\$}(\mu,\nu)=0$.
\end{definition}
\begin{remark}\label{p1remarkdollar}
When $p=1$, by the Kantorovich-Rubinstein theorem (\ref{krthm}), we have $W_{1}^{\$}(\mu,\nu)=W_{1}(\mu,\nu)$.
\end{remark}
\begin{remark}
When $p>1$, the $W_{p}^{\$}$ does not necessarily define a metric on the set of all probability measures on $E$.
\end{remark}
We now prove two lemmas regarding the classical $p$-Wasserstein distance $W_{p}$. These lemmas are then applied to bound $W_{p}^{\$}(\mu,\nu)$ in Lemma \ref{basicbound} below. The first lemma says that the quantity
\[(1-\zeta(E))^{\frac{1}{p}}W_{p}\left(\frac{\mu-\zeta}{1-\zeta(E)},\frac{\nu-\zeta}{1-\zeta(E)}\right)\]
is increasing in $\zeta$ for measures $\zeta\leq\mu\wedge\nu$.
\begin{lemma}\label{dollar2}
Let $\mu$ and $\nu$ be probability measures on a finite metric space $E$. Let $\zeta$ and $\eta$ be measures on $E$ such that $\zeta\leq\eta\leq\mu\wedge\nu$. Then
\[(1-\zeta(E))^{\frac{1}{p}}W_{p}\left(\frac{\mu-\zeta}{1-\zeta(E)},\frac{\nu-\zeta}{1-\zeta(E)}\right)\leq(1-\eta(E))^{\frac{1}{p}}W_{p}\left(\frac{\mu-\eta}{1-\eta(E)},\frac{\nu-\eta}{1-\eta(E)}\right).\]
\end{lemma}
\begin{proof}
Let
\[t=\frac{1-\eta(E)}{1-\zeta(E)}.\]
Note that $t\in[0,1]$. We have
\[\frac{\mu-\zeta}{1-\zeta(E)}=t\left(\frac{\mu-\eta}{1-\eta(E)}\right)+(1-t)\left(\frac{\eta-\zeta}{\eta(E)-\zeta(E)}\right),\]
and
\[\frac{\nu-\zeta}{1-\zeta(E)}=t\left(\frac{\nu-\eta}{1-\eta(E)}\right)+(1-t)\left(\frac{\eta-\zeta}{\eta(E)-\zeta(E)}\right).\]
So applying Lemma \ref{mixlemma} with $k=2$, $t_{1}=t$ and $t_{2}=1-t$, we obtain
\[W_{p}\left(\frac{\mu-\zeta}{1-\zeta(E)},\frac{\nu-\zeta}{1-\zeta(E)}\right)^{p}\leq t\cdot W_{p}\left(\frac{\mu-\eta}{1-\eta(E)},\frac{\nu-\eta}{1-\eta(E)}\right)^{p}+(1-t)\cdot 0.\]
Thus the result follows.
\end{proof}
\begin{remark}\label{wpdollarineq}
Taking $\zeta=0$ and $\eta=\mu\wedge\nu$ in Lemma \ref{dollar2}, we get $W_{p}(\mu,\nu)\leq W_{p}^{\$}(\mu,\nu)$.
\end{remark}
\begin{lemma}\label{transbound1}
Suppose that $(E,d_{E})$ is a finite metric space and $T:E\to E$ is a map. Let $\mu$ and $\nu$ be probability measures on $E$. Then
\[W_{p}(\mu,\nu)\leq W_{p}(T_{\#}\mu,T_{\#}\nu)+\left(\sum_{x\in E}d_{E}(x,T(x))^{p}\mu(\{x\})\right)^{1/p}+\left(\sum_{x\in E}d_{E}(x,T(x))^{p}\nu(\{x\})\right)^{1/p}.\]
\end{lemma}
\begin{proof}
Since $W_{p}$ is a metric, we have
\[W_{p}(\mu,\nu)\leq W_{p}(T_{\#}\mu,T_{\#}\nu)+W_{p}(\mu,T_{\#}\mu)+W_{p}(\nu,T_{\#}\nu).\]
Apply Lemma \ref{mutmu} to bound each of the two terms $W_{p}(\mu,T_{\#}\mu)$ and $W_{p}(\nu,T_{\#}\nu)$. The result follows.
\end{proof}
\begin{lemma}\label{basicbound}
Suppose that $(E,d_{E})$ is a finite metric space and $T:E\to E$ is a map. Let $\mu$ and $\nu$ be probability measures on $E$. Then
\[W_{p}^{\$}(\mu,\nu)\leq W_{p}^{\$}(T_{\#}\mu,T_{\#}\nu)+2^{1-1/p}\left(\sum_{x\in E}d_{E}(x,T(x))^{p}|\mu(\{x\})-\nu(\{x\})|\right)^{1/p}.\]
\end{lemma}
\begin{proof}
Let $\zeta=\mu\wedge\nu$. By Lemma \ref{transbound1},
\begin{align*}
&W_{p}\left(\frac{\mu-\zeta}{1-\zeta(E)},\frac{\nu-\zeta}{1-\zeta(E)}\right)\\\leq&
W_{p}\left(T_{\#}\left(\frac{\mu-\zeta}{1-\zeta(E)}\right),T_{\#}\left(\frac{\nu-\zeta}{1-\zeta(E)}\right)\right)+\left(\sum_{x\in E}d_{E}(x,T(x))^{p}\cdot\frac{\mu(\{x\})-\zeta(\{x\})}{1-\zeta(E)}\right)^{1/p}\\&
+\left(\sum_{x\in E}d_{E}(x,T(x))^{p}\cdot\frac{\nu(\{x\})-\zeta(\{x\})}{1-\zeta(E)}\right)^{1/p}.
\end{align*}
So multiplying this inequality by $(1-\zeta(E))^{1/p}$ and using the definition of $W_{p}^{\$}$, we obtain
\begin{align*}
&W_{p}^{\$}(\mu,\nu)\\=&
(1-\zeta(E))^{\frac{1}{p}}W_{p}\left(\frac{\mu-\zeta}{1-\zeta(E)},\frac{\nu-\zeta}{1-\zeta(E)}\right)\\\leq&
(1-\zeta(E))^{\frac{1}{p}}W_{p}\left(\frac{T_{\#}\mu-T_{\#}\zeta}{1-\zeta(E)},\frac{T_{\#}\nu-T_{\#}\zeta}{1-\zeta(E)}\right)+\left(\sum_{x\in E}d_{E}(x,T(x))^{p}\cdot(\mu(\{x\})-\zeta(\{x\}))\right)^{1/p}\\&
+\left(\sum_{x\in E}d_{E}(x,T(x))^{p}\cdot(\nu(\{x\})-\zeta(\{x\}))\right)^{1/p}\\\leq&
(1-\zeta(E))^{\frac{1}{p}}W_{p}\left(\frac{T_{\#}\mu-T_{\#}\zeta}{1-\zeta(E)},\frac{T_{\#}\nu-T_{\#}\zeta}{1-\zeta(E)}\right)\\&
+2^{1-1/p}\left(\sum_{x\in E}d_{E}(x,T(x))^{p}\cdot(\mu(\{x\})+\nu(\{x\})-2\zeta(\{x\}))\right)^{\frac{1}{p}},
\end{align*}
where in the last inequality, we use the fact that $a^{1/p}+b^{1/p}\leq2^{1-1/p}(a+b)^{1/p}$. Since $\zeta=\mu\wedge\nu$, we have
\begin{eqnarray*}
\mu(\{x\})+\nu(\{x\})-2\zeta(\{x\})&=&\mu(\{x\})+\nu(\{x\})-2\min(\mu(\{x\}),\nu(\{x\}))\\&=&
|\mu(\{x\})-\nu(\{x\})|,
\end{eqnarray*}
for every $x\in E$. It follows that
\begin{eqnarray*}
W_{p}^{\$}(\mu,\nu)&\leq&
(1-\zeta(E))^{\frac{1}{p}}W_{p}\left(\frac{T_{\#}\mu-T_{\#}\zeta}{1-\zeta(E)},\frac{T_{\#}\nu-T_{\#}\zeta}{1-\zeta(E)}\right)\\&&
+2^{1-1/p}\left(\sum_{x\in E}d_{E}(x,T(x))^{p}\cdot|\mu(\{x\})-\nu(\{x\})|\right)^{\frac{1}{p}}.
\end{eqnarray*}
It remains to show that
\begin{equation}\label{basicboundeq1}
(1-\zeta(E))^{\frac{1}{p}}W_{p}\left(\frac{T_{\#}\mu-T_{\#}\zeta}{1-\zeta(E)},\frac{T_{\#}\nu-T_{\#}\zeta}{1-\zeta(E)}\right)\leq W_{p}^{\$}(T_{\#},T_{\#}\nu).
\end{equation}
Since $\zeta\leq\mu$ and $\zeta\leq\nu$, we have $T_{\#}\zeta\leq T_{\#}\mu$ and $T_{\#}\zeta\leq T_{\#}\nu$. So $T_{\#}\zeta\leq(T_{\#}\mu)\wedge(T_{\#}\nu)$. Thus, since $\zeta(E)=(T_{\#}\zeta)(E)$, by Lemma \ref{dollar2},
\begin{align*}
&(1-\zeta(E))^{\frac{1}{p}}W_{p}\left(\frac{T_{\#}\mu-T_{\#}\zeta}{1-\zeta(E)},\frac{T_{\#}\nu-T_{\#}\zeta}{1-\zeta(E)}\right)\\=&
(1-(T_{\#}\zeta)(E))^{\frac{1}{p}}W_{p}\left(\frac{T_{\#}\mu-T_{\#}\zeta}{1-(T_{\#}\zeta)(E)},\frac{T_{\#}\nu-T_{\#}\zeta}{1-(T_{\#}\zeta)(E)}\right)\\\leq&
(1-[(T_{\#}\mu)\wedge(T_{\#}\nu)](E))^{\frac{1}{p}}W_{p}\left(\frac{T_{\#}\mu-[(T_{\#}\mu)\wedge(T_{\#}\nu)]}{1-[(T_{\#}\mu)\wedge(T_{\#}\nu)](E)},\frac{T_{\#}\nu-[(T_{\#}\mu)\wedge(T_{\#}\nu)]}{1-[(T_{\#}\mu)\wedge(T_{\#}\nu)](E)}\right)\\=&
W_{p}^{\$}(T_{\#}\mu,T_{\#}\nu).
\end{align*}
This proves (\ref{basicboundeq1}) and the result follows.
\end{proof}
\begin{remark}
When $p=1$, we have $W_{1}^{\$}(\mu,\nu)=W_{1}(\mu,\nu)$ as noted in Remark \ref{p1remarkdollar}. Lemma \ref{basicbound} has a much shorter proof when $p=1$:
\begin{align*}
&W_{1}(\mu,\nu)-W_{1}(T_{\#}\mu,T_{\#}\nu)\\=&
\sup_{\|f\|_{\mathrm{Lip}}\leq 1}\left(\int_{E}f\,d\mu-\int_{E}f\,d\nu\right)
-\sup_{\|f\|_{\mathrm{Lip}}\leq 1}\left(\int_{E}f\circ T\,d\mu-\int_{E}f\circ T\,d\nu\right)\\\leq&
\sup_{\|f\|_{\mathrm{Lip}}\leq 1}\left(\int_{E}f\,d\mu-\int_{E}f\circ T\,d\mu-\int_{E}f\,d\nu+\int_{E}f\circ T\,d\nu\right)\\=&
\sup_{\|f\|_{\mathrm{Lip}}\leq 1}\left(\sum_{x\in E}[f(x)-f(T(x))]\cdot[\mu(\{x\})-\nu(\{x\})]\right)\\\leq&
\sum_{x\in E}d_{E}(x,T(x))|\mu(\{x\})-\nu(\{x\})|,
\end{align*}
where the suprema above are over all $f:E\to\mathbb{R}$ with $\|f\|_{\mathrm{Lip}}\leq 1$.
\end{remark}
\begin{lemma}\label{chainingwp}
Suppose that $(E,d_{E})$ is a finite metric space, $y_{1},\ldots,y_{2r}\in E$ are distinct and $z_{1},\ldots,z_{r}\in E$. Let $\mu=\frac{1}{2r}\sum_{i=1}^{2r}\delta_{y_{i}}$ and $\widetilde{\mu}=\frac{1}{r}\sum_{i=1}^{r}\delta_{z_{i}}$. Then
\[\|W_{p}^{\$}(\mu_{n},\mu)\|_{L^{p}}\leq
\|W_{p}^{\$}(\widetilde{\mu}_{n},\widetilde{\mu})\|_{L^{p}}+2^{1-1/p}\left(\frac{2r}{n}\right)^{1/(2p)}\cdot W_{p}(\mu,\widetilde{\mu}).\]
\end{lemma}
\begin{proof}
For $1\leq i\leq r$, define $z_{i+r}=z_{i}$. Then $\widetilde{\mu}=\frac{1}{r}\sum_{i=1}^{r}\delta_{z_{i}}=\frac{1}{2r}\sum_{i=1}^{2r}\delta_{z_{i}}$. By Lemma \ref{wppermute}, there exists a permutation $\sigma:\{1,\ldots,2r\}\to\{1,\ldots,2r\}$ such that
\begin{equation}\label{chainingwpeq1}
W_{p}(\mu,\widetilde{\mu})=\left(\frac{1}{2r}\sum_{i=1}^{2r}d_{E}(y_{i},z_{\sigma(i)})^{p}\right)^{1/p}.
\end{equation}
Define $T:E\to E$ by
\begin{equation}\label{chainingwpeq2}
T(x)=\begin{cases}z_{\sigma(i)},&x=y_{i}\text{ for some }i\\x,&\text{Otherwise}\end{cases},
\end{equation}
which is well defined since $y_{1},\ldots,y_{2r}$ are distinct. Note that
\[T_{\#}\mu=\frac{1}{2r}\sum_{i=1}^{2r}\delta_{T(y_{i})}=\frac{1}{2r}\sum_{i=1}^{2r}\delta_{z_{\sigma(i)}}=\frac{1}{2r}\sum_{i=1}^{2r}\delta_{z_{i}}=\widetilde{\mu}.\]
Thus,
\begin{equation}\label{chainingwpeq3}
\mathbb{E}[W_{p}^{\$}(T_{\#}\mu_{n},T_{\#}\mu)^{p}]=\mathbb{E}[W_{p}^{\$}(T_{\#}\mu_{n},\widetilde{\mu})^{p}]=\mathbb{E}[W_{p}^{\$}(\widetilde{\mu}_{n},\widetilde{\mu})^{p}],
\end{equation}
where the second equality follows from the fact that $T_{\#}\mu_{n}$ is an empirical measure of $T_{\#}\mu=\widetilde{\mu}$ by Lemma \ref{emptrans}. By Lemma \ref{basicbound},
\begin{align*}
&\|W_{p}^{\$}(\mu_{n},\mu)\|_{L^{p}}\\\leq&
\|W_{p}^{\$}(T_{\#}\mu_{n},T_{\#}\mu)\|_{L^{p}}+
2^{1-1/p}\left(\sum_{x\in E}d_{E}(x,T(x))^{p}\cdot\mathbb{E}|\mu_{n}(\{x\})-\mu(\{x\})|\right)^{1/p}\\=&
\|W_{p}^{\$}(\widetilde{\mu}_{n},\widetilde{\mu})\|_{L^{p}}+
2^{1-1/p}\left(\sum_{i=1}^{2r}d_{E}(y_{i},z_{\sigma(i)})^{p}\cdot\mathbb{E}|\mu_{n}(\{y_{i}\})-\mu(\{y_{i}\})|\right)^{1/p},
\end{align*}
where the last step follows from (\ref{chainingwpeq3}) and (\ref{chainingwpeq2}). Since $\mu(\{y_{i}\})=\frac{1}{2r}$ for each $i$, by the definition of the empirical measure $\mu_{n}$, we have
\[\mathbb{E}|\mu_{n}(\{y_{i}\})-\mu(\{y_{i}\})|^{2}=\frac{1}{n}\left(\frac{1}{2r}\right)\left(1-\frac{1}{2r}\right)\leq\frac{1}{2nr}.\]
Therefore,
\begin{eqnarray*}
\|W_{p}^{\$}(\mu_{n},\mu)\|_{L^{p}}&\leq&
\|W_{p}^{\$}(\widetilde{\mu}_{n},\widetilde{\mu})\|_{L^{p}}
+2^{1-1/p}\left(\frac{1}{\sqrt{2nr}}\sum_{i=1}^{2r}d_{E}(y_{i},z_{\sigma(i)})^{p}\right)^{1/p}\\&=&
\|W_{p}^{\$}(\widetilde{\mu}_{n},\widetilde{\mu})\|_{L^{p}}
+2^{1-1/p}\left(\frac{2r}{n}\right)^{1/(2p)}\cdot W_{p}(\mu,\widetilde{\mu}),
\end{eqnarray*}
where the last step follows from (\ref{chainingwpeq1}).
\end{proof}
\begin{proof}[Proof of  Theorem \ref{main2}]
Without loss of generality, by enlarging the metric space $E$ (e.g., replace $E$ by $E\times[0,1]$), we may assume that $E$ has no isolated points, since when $E$ is enlarged, $W_{p}(\mu_{n},\mu)$ does not change, while the $b_{2^{k},p}(\mu)$ could only get smaller.

For $r\in\mathbb{N}$, define $\mathcal{D}(E,r)$ to be the set of all probability measures on $E$ of the form $\frac{1}{r}\sum_{i=1}^{r}\delta_{x_{i}}$ for some distinct $x_{1},\ldots,x_{r}\in E$. Observe that since $E$ has no isolated points, any list of points $x_{1},\ldots,x_{r}\in E$ can be perturbed so that they are distinct, so
\begin{equation}\label{brpaltd}
b_{r,p}(\mu)=\inf_{x_{1},\ldots,x_{r}\in E}W_{p}\left(\mu,\frac{1}{r}\sum_{i=1}^{r}\delta_{x_{i}}\right)=\inf_{\nu\in\mathcal{D}(E,r)}W_{p}(\mu,\nu),
\end{equation}
for all $r\in\mathbb{N}$.

Let $k_{0}=\lfloor\log_{2}n\rfloor$. For each $0\leq k\leq k_{0}$, fix $\mu^{(k)}\in\mathcal{D}(E,2^{k})$ and let $\mu_{n}^{(k)}$ be the empirical measure of $\mu^{(k)}$ with sample size $n$.

We pause here for a technical note. Let $E_{0}$ be the union of all the supports of the measures $\mu^{(k)}$ for $0\leq k\leq k_{0}$. Then $E_{0}$ is a finite set and all the $\mu^{(k)}$ and $\mu_{n}^{(k)}$ are supported on $E_{0}$.  Since the notion $W_{p}^{\$}$ in Definition \ref{dollardef} is defined only for probability measures on a finite metric space, while the metric space $E$ here is infinite, in what follows, for probability measures $\nu$ and $\zeta$ on $E$ that are supported on $E_{0}$, we write $W_{p}^{\$}(\nu,\zeta,E_{0})$ to mean $W_{p}^{\$}(\nu,\zeta)$ where we treat $\nu$ and $\zeta$ as probability measures on $E_{0}$.

Since $\mu^{(k+1)}\in\mathcal{D}(E,2^{k+1})$ and $\mu^{(k)}\in\mathcal{D}(E,2^{k})$, by Lemma \ref{chainingwp} with $r=2^{k}$, we have
\begin{align*}
&\|W_{p}^{\$}(\mu_{n}^{(k+1)},\mu^{(k+1)},E_{0})\|_{L^{p}}-\|W_{p}^{\$}(\mu_{n}^{(k)},\mu^{(k)},E_{0})\|_{L^{p}}\\\leq&
2^{1-1/p}\left(\frac{2^{k+1}}{n}\right)^{1/(2p)}\cdot W_{p}(\mu^{(k+1)},\mu^{(k)}),
\end{align*}
for all $0\leq k\leq k_{0}-1$. Summing over $0\leq k\leq k_{0}-1$, we obtain
\begin{align*}
&\|W_{p}^{\$}(\mu_{n}^{(k_{0})},\mu^{(k_{0})},E_{0})\|_{L^{p}}-
\|W_{p}^{\$}(\mu_{n}^{(0)},\mu^{(0)},E_{0})\|_{L^{p}}
\\\leq&\sum_{k=0}^{k_{0}-1}2^{1-1/p}\left(\frac{2^{k+1}}{n}\right)^{1/(2p)}\cdot W_{p}(\mu^{(k+1)},\mu^{(k)}).
\end{align*}
Observe that since $\mu^{(0)}\in\mathcal{D}(E,1)$ is supported on one point, we have $\mu_{n}^{(0)}=\mu^{(0)}$ so $W_{p}^{\$}(\mu_{n}^{(0)},\mu^{(0)},E_{0})=0$. Moreover, since $W_{p}^{\$}$ is always at least $W_{p}$ (see Remark \ref{wpdollarineq}), we have $\|W_{p}(\mu_{n}^{(k_{0})},\mu^{(k_{0})})\|_{L^{p}}\leq\|W_{p}^{\$}(\mu_{n}^{(k_{0})},\mu^{(k_{0})},E_{0})\|_{L^{p}}$. Therefore,
\begin{align*}
&\|W_{p}(\mu_{n}^{(k_{0})},\mu^{(k_{0})})\|_{L^{p}}\\\leq&
2^{1-1/p}\sum_{k=0}^{k_{0}-1}\left(\frac{2^{k+1}}{n}\right)^{1/(2p)}\cdot\left[W_{p}(\mu^{(k+1)},\mu)+W_{p}(\mu^{(k)},\mu)\right]\\=&
2^{1-1/p}\sum_{k=1}^{k_{0}}\left(\frac{2^{k}}{n}\right)^{1/(2p)}\cdot W_{p}(\mu^{(k)},\mu)+
2^{1-1/p}\sum_{k=0}^{k_{0}-1}\left(\frac{2^{k+1}}{n}\right)^{1/(2p)}\cdot W_{p}(\mu^{(k)},\mu)\\\leq&
2\left(\frac{2^{k_{0}}}{n}\right)^{1/(2p)}\cdot W_{p}(\mu^{(k_{0})},\mu)+
4\cdot\sum_{k=0}^{k_{0}-1}\left(\frac{2^{k}}{n}\right)^{1/(2p)}\cdot W_{p}(\mu^{(k)},\mu),
\end{align*}
where in the second step, we split the sum into two sums and then re-index the first sum, and in the last step, we keep the $k=k_{0}$ term in the first sum and then combine the other terms $1\leq k\leq k_{0}-1$ in the first sum with the second sum. But by Lemma \ref{easybound},
\[\|W_{p}(\mu_{n},\mu)\|_{L^{p}}\leq
\|W_{p}(\mu_{n}^{(k_{0})},\mu^{(k_{0})})\|_{L^{p}}+
2W_{p}(\mu,\mu^{(k_{0})}).\]
Therefore,
\[\|W_{p}(\mu_{n},\mu)\|_{L^{p}}\leq
\left(2+2\left(\frac{2^{k_{0}}}{n}\right)^{1/(2p)}\right)W_{p}(\mu^{(k_{0})},\mu)+4\cdot\sum_{k=0}^{k_{0}-1}\left(\frac{2^{k}}{n}\right)^{1/(2p)}\cdot W_{p}(\mu^{(k)},\mu).\]
Since $k_{0}=\lfloor\log_{2}n\rfloor\geq-1+\log_{2}n$, we have $\frac{n}{2^{k_{0}}}\leq 2$ and so the coefficient of the first term on the right hand side
\[2+2\left(\frac{2^{k_{0}}}{n}\right)^{1/(2p)}=\left\{2\left(\frac{n}{2^{k_{0}}}\right)^{1/(2p)}+2\right\}\left(\frac{2^{k_{0}}}{n}\right)^{1/(2p)}\leq(2\sqrt{2}+2)\cdot\left(\frac{2^{k_{0}}}{n}\right)^{1/(2p)}.\]
Thus, since $4\leq2\sqrt{2}+2$, it follows that
\[\|W_{p}(\mu_{n},\mu)\|_{L^{p}}\leq(2\sqrt{2}+2)\cdot\sum_{k=0}^{k_{0}}\left(\frac{2^{k}}{n}\right)^{1/(2p)}\cdot W_{p}(\mu^{(k)},\mu).\]
Taking infimum over all $\mu^{(k)}\in\mathcal{D}(E,2^{k})$, for $0\leq k\leq k_{0}$, and using (\ref{brpaltd}) completes the proof.
\end{proof}
\section{Proof of the first main result}
In this section, we prove the first main result Theorem \ref{main1}. Since the upper bound in Theorem \ref{main1} follows immediately from Theorem \ref{main2}, which has been proved in the previous section, it suffices to prove the lower bound in Theorem \ref{main1}. Observe that if we  apply Lemma \ref{basicstability} with $m=2^{k}$, we obtain
\[\mathbb{E}[W_{1}(\mu_{n},\mu)]\geq\max_{0\leq k\leq\lfloor\log_{2}n\rfloor}\frac{2^{k}}{n}\cdot\mathbb{E}[W_{1}(\mu,\mu_{2^{k}})]\geq
\max_{0\leq k\leq\lfloor\log_{2}n\rfloor}\frac{2^{k}}{n}\cdot b_{2^{k},1}(\mu).\]
This does not suffice to prove the lower bound in Theorem \ref{main1} due to the missing square root on the coefficient $\frac{2^{k}}{n}$. The needed square root on $\frac{2^{k}}{n}$ can be obtained by using Lemma \ref{w1lb} below at the cost of a log factor. Before we can prove Lemma \ref{w1lb}, we need several lemmas, namely, Lemma \ref{Hoeffding} to Lemma \ref{w1prelb}.
\begin{lemma}[Section 2.2 in \cite{Romanbook}, Hoeffding inequality]\label{Hoeffding}
Suppose that $s_{1},\ldots,s_{m}$ are independent Rademacher random variables (i.e., $\mathbb{P}(s_{i}=1)=\mathbb{P}(s_{i}=-1)=\frac{1}{2}$) and $a_{1},\ldots,a_{m}\in\mathbb{R}$. Then
\[\mathbb{P}\left(\sum_{i=1}^{m}s_{i}a_{i}\geq t\right)\leq\exp\left(-\frac{t^{2}}{2\sum_{i=1}^{m}a_{i}^{2}}\right),\]
for all $t\geq 0$.
\end{lemma}
\begin{lemma}\label{Lipbound}
Suppose that $E$ is a finite metric space. Let $f_{i}:E\to\mathbb{R}$, for $1\leq i\leq m$, be such that $\|f_{i}\|_{\mathrm{Lip}}\leq 1$. Let $s_{1},\ldots,s_{m}$ be independent Rademacher random variables. Then
\[\mathbb{E}\left\|\sum_{i=1}^{m}s_{i}f_{i}\right\|_{\mathrm{Lip}}^{2}\leq14m\ln|E|.\]
\end{lemma}
\begin{proof}
For all $x,y\in E$ with $x\neq y$, we have $\left|\frac{f_{i}(x)-f_{i}(y)}{d_{E}(x,y)}\right|\leq\|f_{i}\|_{\mathrm{Lip}}\leq 1$ for each $i$, so by Lemma \ref{Hoeffding},
\[\mathbb{P}\left(\sum_{i=1}^{m}s_{i}\cdot\frac{f_{i}(x)-f_{i}(y)}{d_{E}(x,y)}\geq t\right)\leq\exp\left(-\frac{t^{2}}{2m}\right),\]
for all $t\geq 0$ and $x,y\in E$ with $x\neq y$. Taking a union bound, we have
\begin{align*}
&\mathbb{P}\left(\max_{\substack{x,y\in E\\x\neq y}}\sum_{i=1}^{m}s_{i}\cdot\frac{f_{i}(x)-f_{i}(y)}{d_{E}(x,y)}\geq t\right)\leq|E|^{2}\exp\left(-\frac{t^{2}}{2m}\right)\\=&
\mathbb{P}\left(\left\|\sum_{i=1}^{m}s_{i}f_{i}\right\|_{\mathrm{Lip}}\geq t\right),
\end{align*}
for all $t\geq 0$. So
\begin{eqnarray*}
\mathbb{E}\left\|\sum_{i=1}^{m}s_{i}f_{i}\right\|_{\mathrm{Lip}}^{2}&=&
\int_{0}^{\infty}\mathbb{P}\left(\left\|\sum_{i=1}^{m}s_{i}f_{i}\right\|_{\mathrm{Lip}}\geq\sqrt{t}\right)\,dt\\&\leq&
\int_{0}^{\infty}\min\left(1,\,|E|^{2}\exp\left(-\frac{t}{2m}\right)\right)\,dt\\&=&
8m\cdot\int_{0}^{\infty}\min(1,\,|E|^{2}e^{-4t})\,dt\\&\leq&
8m\cdot\int_{0}^{\ln|E|}1\,dt+8m\cdot\int_{\ln|E|}^{\infty}\,|E|^{2}e^{-4t}\,dt.
\end{eqnarray*}
For $t\geq\ln|E|$, we have $|E|^{2}\leq e^{2t}$ and so $|E|^{2}e^{-4t}\leq e^{-2t}$. Therefore,
\[\mathbb{E}\left\|\sum_{i=1}^{m}s_{i}f_{i}\right\|_{\mathrm{Lip}}^{2}\leq 8m\ln|E|+8m\cdot\int_{0}^{\infty}e^{-2t}\,dt\leq 14m\ln|E|.\]
\end{proof}
\begin{lemma}[\cite{LO}]\label{kk}
If $(V,\|\,\|)$ is a normed space, $L_{1},\ldots,L_{m}\in V$ and $s_{1},\ldots,s_{m}$ are independent Rademacher random variables, then
\[\left(\mathbb{E}\left\|\sum_{i=1}^{m}s_{i}L_{i}\right\|^{2}\right)^{\frac{1}{2}}\leq\sqrt{2}\cdot\mathbb{E}\left\|\sum_{i=1}^{m}s_{i}L_{i}\right\|.\]
\end{lemma}
The next lemma, Lemma \ref{w1prelb}, is needed to lower bound the expectation that arises from the partial symmetrization step in the proof of Lemma \ref{w1lb} below. Lemma \ref{w1prelb} can be thought of as a dual version of Lemma \ref{Lipbound}. While Lemma \ref{Lipbound} gives an upper bound for the Lipschitz seminorm of a random sum, Lemma \ref{w1prelb} gives a lower bound for the supremum of a random sum over all functions $f$ with $\|f\|_{\mathrm{Lip}}\leq 1$. The proof uses a duality argument that resembles the proof of the fact that if a Banach space has type 2 then its dual has cotype 2.
\begin{lemma}\label{w1prelb}
Suppose that $(E,d_{E})$ is a metric space and $x_{i,j}\in E,\,y_{i,j}\in E$ for $1\leq i\leq m$ and $1\leq j\leq n$. Let $s_{1},\ldots,s_{m}$ be independent Rademacher random variables. Then
\[\mathbb{E}\sup_{\|f\|_{\mathrm{Lip}\leq1}}\sum_{i=1}^{m}s_{i}\sum_{j=1}^{n}(f(x_{i,j})-f(y_{i,j}))\geq\frac{1}{\sqrt{28m\ln(2mn)}}\sum_{i=1}^{m}\sup_{\|f\|_{\mathrm{Lip}}\leq 1}\,\sum_{j=1}^{n}(f(x_{i,j})-f(y_{i,j})),\]
where the suprema are over all $f:E\to\mathbb{R}$ with $\|f\|_{\mathrm{Lip}}\leq 1$.
\end{lemma}
\begin{proof}
Let $E_{0}=\cup_{i=1}^{m}\cup_{j=1}^{n}\{x_{i,j},y_{i,j}\}$. Without loss of generality, we may assume that $E=E_{0}$. This is because every $f_{0}:E_{0}\to\mathbb{R}$ with $\|f_{0}\|_{\mathrm{Lip}}\leq 1$ can be extended to some $f:E\to\mathbb{R}$ with $\|f\|_{\mathrm{Lip}}\leq 1$, e.g., take $f(x)=\sup_{y\in E_{0}}(f_{0}(y)-d_{E}(x,y))$ for $x\in E$.

Fix $z\in E$. Let $U$ be the vector space of all functions $f:E\to\mathbb{R}$ such that $f(z)=0$. Note that $f\mapsto\|f\|_{\mathrm{Lip}}$ defines a norm on $U$. Thus $(U,\|\,\|_{\mathrm{Lip}})$ is a normed space. Let $(V,\|\,\|_{*})$ be the dual of this normed space, i.e., $V$ is the vector space consisting of linear maps $L:U\to\mathbb{R}$, and the norm $\|\,\|_{*}$ on $V$ is defined by
\[\|L\|_{*}:=\sup_{f\in U\backslash\{0\}}\frac{L(f)}{\|f\|_{\mathrm{Lip}}}=\sup_{\substack{f\in U\\\|f\|_{\mathrm{Lip}\leq 1}}}L(f).\]
For each $1\leq i\leq m$, define $L_{i}\in V$ by
\[L_{i}(f)=\sum_{j=1}^{n}(f(x_{i,j})-f(y_{i,j})),\quad f\in U.\]
We need to prove that
\begin{equation}\label{w1prelbeq1}
\mathbb{E}\left\|\sum_{i=1}^{m}s_{i}L_{i}\right\|_{*}\geq
\frac{1}{\sqrt{28m\ln(2mn)}}\sum_{i=1}^{m}\|L_{i}\|_{*}.
\end{equation}
For all $f_{1},\ldots,f_{n}\in U$ such that $\|f_{i}\|_{\mathrm{Lip}}\leq 1$ for all $i$, we have
\begin{eqnarray*}
\sum_{i=1}^{m}L_{i}(f_{i})&=&\mathbb{E}\left(\sum_{i=1}^{m}s_{i}L_{i}\right)\left(\sum_{k=1}^{m}s_{i}f_{i}\right)\\&\leq&
\mathbb{E}\left\|\sum_{i=1}^{m}s_{i}L_{i}\right\|_{*}\left\|\sum_{k=1}^{m}s_{i}f_{i}\right\|_{\mathrm{Lip}}\\&\leq&
\left(\mathbb{E}\left\|\sum_{i=1}^{m}s_{i}L_{i}\right\|_{*}^{2}\right)^{\frac{1}{2}}\cdot
\left(\mathbb{E}\left\|\sum_{k=1}^{m}s_{i}f_{i}\right\|_{\mathrm{Lip}}^{2}\right)^{\frac{1}{2}}\\&\leq&
\sqrt{2}\cdot\left(\mathbb{E}\left\|\sum_{i=1}^{m}s_{i}L_{i}\right\|_{*}\right)\cdot\left(14m\ln|E|\right)^{\frac{1}{2}},
\end{eqnarray*}
where the last inequality follows from Lemma \ref{kk} and Lemma \ref{Lipbound}. Taking supremum over all $f_{1},\ldots,f_{n}\in U$ such that $\|f_{i}\|_{\mathrm{Lip}}\leq 1$ for all $i$, we obtain
\[\sum_{i=1}^{m}\|L_{i}\|_{*}\leq\sqrt{28m\ln|E|}\cdot\mathbb{E}\left\|\sum_{i=1}^{m}s_{i}L_{i}\right\|_{*}.\]
Since $|E|=|E_{0}|\leq 2mn$, this proves (\ref{w1prelbeq1}). So the result follows, since for every function $f:E\to\mathbb{R}$, we have $f-f(z)\in U$ and $\|f\|_{\mathrm{Lip}}=\|f-f(z)\|_{\mathrm{Lip}}$.
\end{proof}
\begin{lemma}\label{w1lb}
Let $E$ be a separable metric space and $\mu\in\mathcal{P}_{1}(E)$. Then for all $m,n\in\mathbb{N}$,
\[\mathbb{E}[W_{1}(\mu_{mn},\mu)]\geq\frac{1}{11\sqrt{m\ln(2mn)}}\cdot\mathbb{E}[W_{1}(\mu_{n},\mu)].\]
\end{lemma}
\begin{proof}
Let $X_{i,j}$ and $Y_{i,j}$, for $1\leq i\leq m$ and $1\leq j\leq n$, be independent samples of $\mu$. We begin with a partial symmetrization step.

{\bf Claim:} For all fixed choices $s_{1},\ldots,s_{m}\in\{\pm 1\}$, the two random variables
\[\sup_{\|f\|_{\mathrm{Lip}}\leq 1}\sum_{i=1}^{m}\sum_{j=1}^{n}(f(X_{i,j})-f(Y_{i,j}))\quad\text{and}\quad\sup_{\|f\|_{\mathrm{Lip}}\leq 1}\sum_{i=1}^{m}s_{i}\sum_{j=1}^{n}(f(X_{i,j})-f(Y_{i,j}))\]
have the same distribution, where the suprema are over all $f:E\to\mathbb{R}$ with $\|f\|_{\mathrm{Lip}}\leq 1$. Let us first prove the claim for $s_{1}=-1$, $s_{2},\ldots,s_{n}=1$. Define
\[H((X_{i,j})_{i,j},(Y_{i,j})_{i,j}):=\sup_{\|f\|_{\mathrm{Lip}}\leq 1}\sum_{i=1}^{m}\sum_{j=1}^{n}(f(X_{i,j})-f(Y_{i,j})).\]
If one interchanges the roles of $(X_{1,j})_{j}$ and $(Y_{1,j})_{j}$, the distribution of $H((X_{i,j})_{i,j},(Y_{i,j})_{i,j})$ is still the same and yet after interchanging those roles, $H((X_{i,j})_{i,j},(Y_{i,j})_{i,j})$ becomes
\[\sup_{\|f\|_{\mathrm{Lip}}\leq 1}\sum_{i=1}^{m}s_{i}\sum_{j=1}^{n}(f(X_{i,j})-f(Y_{i,j})),\]
where $s_{1}=-1$ and $s_{2},\ldots,s_{m}=1$. For general choices of $s_{1},\ldots,s_{m}\in\{\pm 1\}$, let $I=\{i:\,s_{i}=-1\}$. For each $i\in I$, interchange the roles of $(X_{i,j})_{j}$ and $(Y_{i,j})_{j}$. But for $i\notin I$, keep $(X_{i,j})_{j}$ and $(Y_{i,j})_{j}$ the same. Then after this interchange process, the distribution of $H((X_{i,j})_{i,j},(Y_{i,j})_{i,j})$ is still the same and yet $H((X_{i,j})_{i,j},(Y_{i,j})_{i,j})$ becomes
\[\sup_{\|f\|_{\mathrm{Lip}}\leq 1}\sum_{i=1}^{m}s_{i}\sum_{j=1}^{n}(f(X_{i,j})-f(Y_{i,j})).\]
This proves the claim. Using the claim, we have
\begin{equation}\label{w1lbeq2}
\mathbb{E}\sup_{\|f\|_{\mathrm{Lip}}\leq 1}\sum_{i=1}^{m}\sum_{j=1}^{n}(f(X_{i,j})-f(Y_{i,j}))=\mathbb{E}\sup_{\|f\|_{\mathrm{Lip}}\leq 1}\sum_{i=1}^{m}s_{i}\sum_{j=1}^{n}(f(X_{i,j})-f(Y_{i,j})),
\end{equation}
for all $s_{1},\ldots,s_{m}\in\{\pm 1\}$. Since this holds for all $s_{1},\ldots,s_{m}\in\{\pm 1\}$, we can randomize $s_{1},\ldots,s_{m}$ to be independent Rademacher random variables that are independent of all the $X_{i,j}$ and $Y_{i,j}$. After this randomization of $s_{1},\ldots,s_{m}$, the equation (\ref{w1lbeq2}) still holds but with the expectation on the right hand side being over all $s_{1},\ldots,s_{m}$ and all $X_{i,j}$ and $Y_{i,j}$.

We now lower bound the right hand side of (\ref{w1lbeq2}). Applying Lemma \ref{w1prelb} conditioning on all the $X_{i,j}$ and $Y_{i,j}$, we have
\begin{align*}
&\mathbb{E}\left(\left.\sup_{\|f\|_{\mathrm{Lip}}\leq 1}\sum_{i=1}^{m}s_{i}\sum_{j=1}^{n}(f(X_{i,j})-f(Y_{i,j}))\right|\,(X_{i,j})_{i,j},\,(Y_{i,j})_{i,j}\right)\\\geq&
\frac{1}{\sqrt{28m\ln(2mn)}}\sum_{i=1}^{m}\sup_{\|f\|_{\mathrm{Lip}}\leq 1}\,\sum_{j=1}^{n}(f(X_{i,j})-f(Y_{i,j})).
\end{align*}
Taking the full expectation, we have
\begin{align*}
&\mathbb{E}\sup_{\|f\|_{\mathrm{Lip}}\leq 1}\sum_{i=1}^{m}s_{i}\sum_{j=1}^{n}(f(X_{i,j})-f(Y_{i,j}))\\\geq&
\frac{1}{\sqrt{28m\ln(2mn)}}\sum_{i=1}^{m}\mathbb{E}\sup_{\|f\|_{\mathrm{Lip}}\leq 1}\,\sum_{j=1}^{n}(f(X_{i,j})-f(Y_{i,j}))\\=&
\frac{m}{\sqrt{28m\ln(2mn)}}\cdot\mathbb{E}\sup_{\|f\|_{\mathrm{Lip}}\leq 1}\,\sum_{j=1}^{n}(f(X_{1,j})-f(Y_{1,j}))\\=&
\frac{m\cdot n}{\sqrt{28m\ln(2mn)}}\cdot\mathbb{E}\left[W_{1}\left(\frac{1}{n}\sum_{j=1}^{n}\delta_{X_{1,j}},\frac{1}{n}\sum_{j=1}^{n}\delta_{Y_{1,j}}\right)\right]\\\geq&
\frac{m\cdot n}{\sqrt{28m\ln(2mn)}}\cdot\mathbb{E}[W_{1}(\mu_{n},\mu)],
\end{align*}
where the second step follows from the fact that the for each $i$, the joint distribution of $(X_{i,j},Y_{i,j})_{1\leq j\leq n}$ is the same as the joint distribution of $(X_{1,j},Y_{1,j})_{1\leq j\leq n}$, the third step follows from (\ref{krthm}), and the last inequality follows from Lemma \ref{2samples} with $p=1$. Thus we have found a lower bound for the right hand side of (\ref{w1lbeq2}). So by the definition of the 1-Wasserstein distance and by (\ref{w1lbeq2}), it follows that
\begin{eqnarray*}
\mathbb{E}\left[W_{1}\left(\frac{1}{m\cdot n}\sum_{i,j}\delta_{X_{i,j}},\frac{1}{m\cdot n}\sum_{i,j}\delta_{Y_{i,j}}\right)\right]&=&
\frac{1}{m\cdot n}\mathbb{E}\sup_{\|f\|_{\mathrm{Lip}}\leq 1}\sum_{i=1}^{m}\sum_{j=1}^{n}(f(X_{i,j})-f(Y_{i,j}))\\&=&
\frac{1}{m\cdot n}\mathbb{E}\sup_{\|f\|_{\mathrm{Lip}}\leq 1}\sum_{i=1}^{m}s_{i}\sum_{j=1}^{n}(f(X_{i,j})-f(Y_{i,j}))\\&\geq&
\frac{1}{\sqrt{28m\ln(2mn)}}\cdot\mathbb{E}[W_{1}(\mu_{n},\mu)].
\end{eqnarray*}
By Lemma \ref{2samples}, $\mathbb{E}\left[W_{1}\left(\frac{1}{mn}\sum_{i,j}\delta_{X_{i,j}},\frac{1}{mn}\sum_{i,j}\delta_{Y_{i,j}}\right)\right]\leq 2\cdot\mathbb{E}[W_{1}(\mu_{mn},\mu)]$. Thus the result follows.
\end{proof}
\begin{proof}[Proof of the lower bound in Theorem \ref{main1}]
Let $k_{0}=\lfloor\log_{2}n\rfloor$. By Lemma \ref{basicstability},
\begin{equation}\label{main1lbeq1}
\mathbb{E}[W_{1}(\mu_{n},\mu)]\geq\frac{2^{k_{0}}}{n}\cdot\mathbb{E}[W_{1}(\mu_{2^{k_{0}}},\mu)]\geq\frac{1}{2}\cdot\mathbb{E}[W_{1}(\mu_{2^{k_{0}}},\mu)].
\end{equation}
Fix $0\leq k\leq k_{0}$. Write $2^{k_{0}}=2^{k_{0}-k}2^{k}$. By Lemma \ref{w1lb}, we have
\begin{eqnarray*}
\mathbb{E}[W_{1}(\mu_{2^{k_{0}}},\mu)]&\geq&\frac{1}{11\sqrt{2^{k_{0}-k}\ln(2\cdot 2^{k_{0}})}}\cdot\mathbb{E}[W_{1}(\mu_{2^{k}},\mu)]\\&=&
\frac{1}{11\sqrt{\ln(2\cdot 2^{k_{0}})}}\cdot\sqrt{\frac{2^{k}}{2^{k_{0}}}}\cdot\mathbb{E}[W_{1}(\mu_{2^{k}},\mu)]\\&\geq&
\frac{1}{11\sqrt{\ln(2n)}}\cdot\sqrt{\frac{2^{k}}{n}}\cdot\mathbb{E}[W_{1}(\mu_{2^{k}},\mu)]\geq
\frac{1}{11\sqrt{\ln(2n)}}\cdot\sqrt{\frac{2^{k}}{n}}\cdot b_{2^{k},1}(\mu).
\end{eqnarray*}
By (\ref{main1lbeq1}), the lower bound in Theorem \ref{main1} follows.
\end{proof}
\section{Proof of the third main result}
In this section, we prove the third main result Theorem \ref{main3}. To bound any quantity of the form $b_{n,p}(\mu)$, one needs to construct a measure of uniform type (i.e., a probability measure of the form $\frac{1}{n}\sum_{i=1}^{n}\delta_{x_{i}}$) to approximate $\mu$. But how to construct a measure of uniform type that gives a good approximation of $\mu$? In Lemma \ref{advdecomp} below, we show how to decompose $\mu$ as a weighted sum of probability measures each of which has a pushforward that is of uniform type. As shown in (\ref{proofmain3eq3}) in the proof of Theorem \ref{main3} below, the weighted sum of these pushforward measures happen to be also of uniform type. By Lemma \ref{mixlemma}, we can then use this weighted sum of pushforward measures to approximate $\mu$ so that we can bound $b_{2^{k},p}(\mu)$ in Theorem \ref{main3}.

Throughout this section, for a set $E$ and $r\in\mathbb{N}$, let $\mathcal{U}(E,r)$ be the set of all probability measures on $E$ of the form $\frac{1}{r}\sum_{i=1}^{r}\delta_{x_{i}}$ for some $x_{1},\ldots,x_{r}\in E$. Observe that when $E$ is a metric space and $\mu\in\mathcal{P}_{p}(E)$, we have
\begin{equation}\label{bnralternative}
b_{r,p}(\mu)=\inf_{x_{1},\ldots,x_{r}\in E}W_{p}\left(\mu,\frac{1}{r}\sum_{i=1}^{r}\delta_{x_{i}}\right)=\inf_{\nu\in\mathcal{U}(E,r)}W_{p}(\mu,\nu),
\end{equation}
for all $r\in\mathbb{N}$ and $p\geq 1$.
\begin{lemma}\label{axiomchoice}
Suppose that $S$ is a finite set and $r\in\mathbb{N}$. Let $f:S\to\mathbb{N}\cup\{0\}$ be such that $\sum_{y\in S}f(y)\geq r$. Then there exists $g:S\to\mathbb{N}\cup\{0\}$ such that $\sum_{y\in S}g(y)=r$ and $g(y)\leq f(y)$ for all $y\in S$.
\end{lemma}
\begin{proof}
If $\sum_{y\in S}f(y)=r$, then take $g=f$ and done. If not, then $\sum_{y\in S}f(y)>r$, so there exists $y_{0}\in S$ such that $f(y_{0})\geq 1$; now replace $f(y_{0})$ by $f(y_{0})-1$, i.e., redefine its value. Repeat this process until $\sum_{y\in S}f(y)=r$. Then take $g=f$. This guarantees that $\sum_{y\in S}g(y)=r$ and that for each $y\in S$, the value of $g(y)$ is at most the original value of $f(y)$.
\end{proof}
\begin{lemma}\label{basicdecomp}
Suppose that $r\in\mathbb{N}$ and $E$ is a measurable space. Let $T:E\to E$ be a measurable map whose range $\mathrm{ran}(T)$ contains at most $r$ points. Then every probability measure $\mu$ on $E$ can be decomposed as
\[\mu=\frac{1}{2}\lambda+\frac{1}{2}\zeta,\]
for some probability measures $\lambda$ and $\zeta$ on $E$ such that $T_{\#}\lambda\in\mathcal{U}(E,r)$.
\end{lemma}
\begin{proof}
Let $S=\{y\in\mathrm{ran}(T):\,\mu(T^{-1}\{y\})\neq 0\}$. Since $\mathrm{ran}(T)$ is finite, $S$ is also finite. Moreover, $\mu(T^{-1}\{y\})=0$ for every $y\in\mathrm{ran}(T)\backslash S$. So
\begin{equation}\label{basicdecompeq1}
\sum_{y\in S}\mu(T^{-1}\{y\})=\sum_{y\in\mathrm{ran}(T)}\mu(T^{-1}\{y\})=\mu(E)=1.
\end{equation}
Define $f:S\to\mathbb{N}\cup\{0\}$ by
\[f(y)=\lfloor 2r\cdot\mu(T^{-1}\{y\})\rfloor,\quad y\in S.\]
Let $|S|$ be the cardinality of $S$. By assumption, $|S|\leq|\mathrm{ran}(T)|\leq r$. We have
\[\sum_{y\in S}f(y)\geq\sum_{y\in S}(2r\cdot\mu(T^{-1}\{y\})-1)=2r\sum_{y\in S}\mu(T^{-1}\{y\})-|S|=2r-|S|\geq r,\]
where we use (\ref{basicdecompeq1}) in the third step. So by Lemma \ref{axiomchoice}, there exists $g:S\to\mathbb{N}\cup\{0\}$ such that $\sum_{y\in S}g(y)=r$ and $g(y)\leq f(y)\leq 2r\cdot\mu(T^{-1}\{y\})$ for all $y\in S$.

For each $y\in S$, define the probability measure $\mu^{(y)}$ on $E$ by
\[\mu^{(y)}(A)=\frac{\mu(A\cap T^{-1}\{y\})}{\mu(T^{-1}\{y\})},\]
for all measurable $A\subset E$. Define the probability measures $\lambda$ and $\zeta$ on $E$ by
\[\lambda=\sum_{y\in S}\frac{g(y)}{r}\cdot\mu^{(y)},\]
and
\[\zeta=\sum_{y\in S}\frac{2r\cdot\mu(T^{-1}\{y\})-g(y)}{r}\cdot\mu^{(y)}.\]
It is easy to check that $\lambda$ and $\zeta$ are indeed probability measures, since $\sum_{y\in S}g(y)=r$ and $0\leq g(y)\leq 2r\cdot\mu(T^{-1}\{y\})$ for all $y\in S$, and since $\sum_{y\in S}\mu(T^{-1}\{y\})=1$ by (\ref{basicdecompeq1}). We have
\[\frac{1}{2}\lambda+\frac{1}{2}\zeta=\sum_{y\in S}\mu(T^{-1}\{y\})\cdot\mu^{(y)}=\mu,\]
so $\mu=\frac{1}{2}\lambda+\frac{1}{2}\zeta$. Finally, since $T_{\#}\mu^{(y)}=\delta_{y}$ for each $y\in S$, we have
\[T_{\#}\lambda=\sum_{y\in S}\frac{g(y)}{r}\cdot T_{\#}\mu^{(y)}=\sum_{y\in S}\frac{g(y)}{r}\cdot\delta_{y},\]
and since each $g(y)\in\mathbb{N}\cup\{0\}$, it follows that $T_{\#}\lambda\in\mathcal{U}(E,r)$.
\end{proof}
\begin{remark}
The idea of the proof of Lemma \ref{basicdecomp} is as follows. Since $\mathrm{ran}(T)$ contains at most $r$ points, $T_{\#}\mu$ is a probability on a set of at most $r$ points. But by taking Lemma \ref{axiomchoice}, for any given probability measure on a set of at most $r$ points, one can construct a decomposition of this measure as $\frac{1}{2}$ times a probability measure, whose mass at each point is an integer multiple of $\frac{1}{r}$, plus $\frac{1}{2}$ times another probability measure. One can then ``inverse pushforward" these two probability measures to construct $\lambda$ and $\zeta$.
\end{remark}
In Lemma \ref{basicdecomp}, after we decompose $\mu$ as $\frac{1}{2}\lambda+\frac{1}{2}\zeta$, we can apply Lemma \ref{basicdecomp} again to decompose $\zeta$. Repeating this decomposition process multiple times, we obtain the following lemma.
\begin{lemma}\label{advdecomp}
Let $k\in\mathbb{N}\cup\{0\}$ and let $E$ be a measurable space set. For each $0\leq j\leq k$, let $T_{j}:E\to E$ be a measurable map whose range contains at most $2^{j}$ points. Then every probability measure $\mu$ on $E$ can be decomposed as
\begin{equation}\label{advdecompeq1}
\mu=\frac{1}{2^{k}}\lambda^{(0)}+\sum_{j=1}^{k}\frac{1}{2^{k-j+1}}\lambda^{(j)},
\end{equation}
for some probability measures $\lambda^{(0)},\ldots,\lambda^{(k)}$ on $E$ such that $(T_{j})_{\#}\lambda^{(j)}\in\mathcal{U}(E,2^{j})$ for every $0\leq j\leq k$.
\end{lemma}
\begin{proof}
We prove this result by induction on $k$. We first prove it for $k=0$ (in which case the sum in (\ref{advdecompeq1}) is void). Since the range of $T_{0}$ is a singleton $\{x_{0}\}$, we have $(T_{0})_{\#}\mu=\delta_{x_{0}}\in\mathcal{U}(E,1)$. Take $\lambda^{(0)}=\mu$. This proves the result for $k=0$.

Assume that the statement of the result is true for a fixed $k$. Let us prove the statement with $k$ being replaced by $k+1$. By Lemma \ref{basicdecomp} with $T=T_{k+1}$ and $r=2^{k+1}$, we can decompose $\mu$ as
\[\mu=\frac{1}{2}\lambda+\frac{1}{2}\zeta,\]
for some probability measures $\lambda$ and $\zeta$ on $E$ such that $(T_{k+1})_{\#}\lambda\in\mathcal{U}(E,2^{k+1})$. Apply the induction hypothesis to $\zeta$. We can decompose $\zeta$ as
\[\zeta=\frac{1}{2^{k}}\lambda^{(0)}+\sum_{j=1}^{k}\frac{1}{2^{k-j+1}}\lambda^{(j)},\]
for some probability measures $\lambda^{(0)},\ldots,\lambda^{(k)}$ such that $(T_{j})_{\#}\lambda^{(j)}\in\mathcal{U}(E,2^{j})$ for every $0\leq j\leq k$. Thus, we have
\begin{eqnarray*}
\mu=\frac{1}{2}\lambda+\frac{1}{2}\zeta&=&
\frac{1}{2}\lambda+\frac{1}{2^{k+1}}\lambda^{(0)}+\sum_{j=1}^{k}\frac{1}{2^{(k+1)-j+1}}\lambda^{(j)}\\&=&
\frac{1}{2^{k+1}}\lambda^{(0)}+\left(\frac{1}{2}\lambda+\sum_{j=1}^{k}\frac{1}{2^{(k+1)-j+1}}\lambda^{(j)}\right).
\end{eqnarray*}
Take $\lambda^{(k+1)}=\lambda$. Then
\[\mu=\frac{1}{2^{k+1}}\lambda^{(0)}+\sum_{j=1}^{k+1}\frac{1}{2^{(k+1)-j+1}}\lambda^{(j)}.\]
This completes the proof.
\end{proof}
\begin{proof}[Proof of Theorem \ref{main3}]
For each $0\leq j\leq k$, fix a finite subset $S_{j}$ of $E$ containing at most $2^{j}$ points. By Lemma \ref{measselect}, for each $0\leq j\leq k$, there is a mesurable map $T_{j}:E\to S_{j}$ such that
\begin{equation}\label{proofmain3eq1}
d_{E}(x,T_{j}(x))=d_{E}(x,S_{j}),
\end{equation}
for all $x\in E$. Applying Lemma \ref{advdecomp} to the maps $T_{j}$ for $0\leq j\leq k$, we can decompose $\mu$ as
\begin{equation}\label{proofmain3eq2}
\mu=\frac{1}{2^{k}}\lambda^{(0)}+\sum_{j=1}^{k}\frac{1}{2^{k-j+1}}\lambda^{(j)},
\end{equation}
for some probability measures $\lambda^{(0)},\ldots,\lambda^{(k)}$ on $E$ such that $(T_{j})_{\#}\lambda^{(j)}\in\mathcal{U}(E,2^{j})$ for every $0\leq j\leq k$. Observe that since $(T_{j})_{\#}\lambda^{(j)}\in\mathcal{U}(E,2^{j})$, the mass of this measure at each point is an integer multiple of $\frac{1}{2^{j}}$, and so the probability measure
\begin{equation}\label{proofmain3eq3}
\frac{1}{2^{k}}(T_{0})_{\#}\lambda^{(0)}+\sum_{j=1}^{k}\frac{1}{2^{k-j+1}}(T_{j})_{\#}\lambda^{(j)}
\end{equation}
is in $\mathcal{U}(E,2^{k+1})$, i.e., it is a discrete measure whose mass at each point is an integer multiple of $\frac{1}{2^{k+1}}$. So by the definition of $b_{2^{k+1},p}(\mu)$, (\ref{proofmain3eq2}) and Lemma \ref{mixlemma}, we have
\begin{eqnarray*}
b_{2^{k+1},p}(\mu)^{p}&\leq&W_{p}\left(\mu,\,\frac{1}{2^{k}}(T_{0})_{\#}\lambda^{(0)}+\sum_{j=1}^{k}\frac{1}{2^{k-j+1}}(T_{j})_{\#}\lambda^{(j)}\right)^{p}\\&\leq&
\frac{1}{2^{k}}W_{p}(\lambda^{(0)},(T_{0})_{\#}\lambda^{(0)})^{p}+\sum_{j=1}^{k}\frac{1}{2^{k-j+1}}W_{p}(\lambda^{(j)},(T_{j})_{\#}\lambda^{(j)})^{p}\\&\leq&
2\sum_{j=0}^{k}\frac{1}{2^{k-j+1}}W_{p}(\lambda^{(j)},(T_{j})_{\#}\lambda^{(j)})^{p},
\end{eqnarray*}
where the last sum is over all $0\leq j\leq k$ rather than $1\leq j\leq k$. For each $0\leq j\leq k$, by Lemma \ref{mutmu} and by (\ref{proofmain3eq1}),
\[W_{p}(\lambda^{(j)},(T_{j})_{\#}\lambda^{(j)})\leq\left(\int_{E}d_{E}(x,T_{j}(x))^{p}\,d\lambda^{(j)}(x)\right)^{1/p}=\left(\int_{E}d_{E}(x,S_{j})^{p}\,d\lambda^{(j)}(x)\right)^{1/p}.\]
Therefore,
\[b_{2^{k+1},p}(\mu)^{p}\leq2\sum_{j=0}^{k}\frac{1}{2^{k-j+1}}\int_{E}d_{E}(x,S_{j})^{p}\,d\lambda^{(j)}(x).\]
Since
\[\left(\int_{E}d_{E}(x,S_{j})^{p}\,d\lambda^{(j)}(x)\right)^{1/p}\leq\left(\int_{E}d_{E}(x,S_{j})^{q}\,d\lambda^{(j)}(x)\right)^{1/q},\]
for every $q\in[p,\infty)$, it follows that
\begin{eqnarray*}
b_{2^{k+1},p}(\mu)^{p}&\leq&
2\sum_{j=0}^{k}\frac{1}{2^{k-j+1}}\inf_{q\in[p,\infty)}\left(\int_{E}d_{E}(x,S_{j})^{q}\,d\lambda^{(j)}(x)\right)^{p/q}\\&=&
2\sum_{j=0}^{k}\inf_{q\in[p,\infty)}2^{(\frac{p}{q}-1)(k-j+1)}\left(\frac{1}{2^{k-j+1}}\int_{E}d_{E}(x,S_{j})^{q}\,d\lambda^{(j)}(x)\right)^{p/q}.
\end{eqnarray*}
But by (\ref{proofmain3eq2}), for each $j$, we have $\frac{1}{2^{k-j+1}}\lambda^{(j)}\leq\mu$, so
\[\frac{1}{2^{k-j+1}}\int_{E}d_{E}(x,S_{j})^{q}\,d\lambda^{(j)}(x)\leq\int_{E}d_{E}(x,S_{j})^{q}\,d\mu(x),\]
for every $0\leq j\leq k$. Thus, it follows that
\[b_{2^{k+1},p}(\mu)^{p}\leq2\sum_{j=0}^{k}\inf_{q\in[p,\infty)}2^{(\frac{p}{q}-1)(k-j+1)}\left(\int_{E}d_{E}(x,S_{j})^{q}\,d\mu(x)\right)^{p/q}.\]
Let $F(S_{j}):=\inf_{q\in[p,\infty)}2^{(\frac{p}{q}-1)(k-j+1)}\left(\int_{E}d_{E}(x,S_{j})^{q}\,d\mu(x)\right)^{p/q}$. Then $b_{2^{k+1},p}(\mu)^{p}\leq2\sum_{j=0}^{k}F(S_{j})$ and so taking infimum over all $S_{j}\subset E$ that contains at most $2^{j}$ points, for $0\leq j\leq k$, and using Lemma \ref{enpalternative}, we obtain
\[b_{2^{k+1},p}(\mu)^{p}\leq2\inf_{S_{0},\ldots,S_{k}}\sum_{j=0}^{k}F(S_{j})=2\sum_{j=0}^{k}\inf_{S_{j}}F(S_{j})=2\sum_{j=0}^{k}\inf_{q\in[p,\infty)}2^{(\frac{p}{q}-1)(k-j+1)}e_{2^{j},q}(\mu)^{p}.\]
Since this holds every $k\in\mathbb{N}\cup\{0\}$, by replacing $k$ by $k-1$, we have
\[b_{2^{k},p}(\mu)^{p}\leq2\sum_{j=0}^{k-1}\inf_{q\in[p,\infty)}2^{(\frac{p}{q}-1)(k-j)}e_{2^{j},q}(\mu)^{p},\]
for all $k\in\mathbb{N}$. This proves the result when $k\in\mathbb{N}$. When $k=0$, the result follows from the fact that $b_{1,p}(\mu)=e_{1,p}(\mu)$.
\end{proof}
\section{Proofs of the remaining results}\label{sectionremaining}
In this section, we prove Corollary \ref{main2corollary}, Corollary \ref{main4}, Corollary \ref{supmu} and Corollary \ref{main4beta}.
\begin{lemma}\label{2a}
For every $0\leq t\leq1$, we have
\[\frac{1}{1-2^{-t}}\leq\frac{2}{t}.\]
\end{lemma}
\begin{proof}
By convexity,
\[2^{-t}=2^{-(t\cdot 1+(1-t)\cdot 0)}\leq t\cdot 2^{-1}+(1-t)\cdot 2^{0}=1-\frac{t}{2}.\]
So the result follows.
\end{proof}
\begin{proof}[Proof of Corollary \ref{main2corollary}]
We first prove (\ref{main2corollaryeq2}). The lower bound follows from $\mathbb{E}[W_{p}(\mu_{n},\mu)]\geq b_{n,p}(\mu)$. As for the upper bound, let $a=\limsup_{n\to\infty}n^{\beta}\cdot b_{n,p}(\mu)$ (which could be $\infty$) and fix $\epsilon>0$. Then there exists $N\in\mathbb{N}$ such that $b_{n,p}(\mu)\leq(a+\epsilon)n^{-\beta}$ for all $n\geq N$. So for all $n\geq 2N$,
\begin{eqnarray*}
\sum_{k=\lfloor\log_{2}N\rfloor+1}^{\lfloor\log_{2}n\rfloor}\left(\frac{2^{k}}{n}\right)^{1/(2p)}\cdot b_{2^{k},p}(\mu)&\leq&
(a+\epsilon)\sum_{k=\lfloor\log_{2}N\rfloor+1}^{\lfloor\log_{2}n\rfloor}\left(\frac{2^{k}}{n}\right)^{1/(2p)}\cdot(2^{k})^{-\beta}\\&\leq&
(a+\epsilon)n^{-1/(2p)}\cdot\sum_{k=-\infty}^{\lfloor\log_{2}n\rfloor}2^{k(\frac{1}{2p}-\beta)}\\&=&
(a+\epsilon)n^{-1/(2p)}\cdot 2^{\lfloor\log_{2}n\rfloor(\frac{1}{2p}-\beta)}/(1-2^{\beta-\frac{1}{2p}})\\&\leq&
(a+\epsilon)n^{-\beta}/(1-2^{\beta-\frac{1}{2p}})\\&\leq&
2(a+\epsilon)n^{-\beta}/\left(\frac{1}{2p}-\beta\right),
\end{eqnarray*}
where the second step follows by extending the sum to start from $k=-\infty$, the third step follows from the fact that it is a geometric series, the fourth step follows by using $\lfloor\log_{2}n\rfloor\leq\log_{2}n$, and the last step follows from Lemma \ref{2a}. So by Theorem \ref{main2},
\begin{eqnarray*}
\left(\mathbb{E}[W_{p}(\mu_{n},\mu)^{p}]\right)^{1/p}&\leq&5\cdot\sum_{k=0}^{\lfloor\log_{2}n\rfloor}\left(\frac{2^{k}}{n}\right)^{1/(2p)}\cdot b_{2^{k},p}(\mu)\\&=&
5\cdot\sum_{k=0}^{\lfloor\log_{2}N\rfloor}\left(\frac{2^{k}}{n}\right)^{1/(2p)}\cdot b_{2^{k},p}(\mu)+5\cdot\sum_{k=\lfloor\log_{2}N\rfloor+1}^{\lfloor\log_{2}n\rfloor}\left(\frac{2^{k}}{n}\right)^{1/(2p)}\cdot b_{2^{k},p}(\mu)
\\&\leq&
5\cdot\sum_{k=0}^{\lfloor\log_{2}N\rfloor}\left(\frac{2^{k}}{n}\right)^{1/(2p)}\cdot b_{2^{k},p}(\mu)+10(a+\epsilon)n^{-\beta}/\left(\frac{1}{2p}-\beta\right).
\end{eqnarray*}
Multiplying both sides by $n^{\beta}$ and taking $\limsup_{n\to\infty}$ completes the proof of (\ref{main2corollaryeq2}).

To prove (\ref{main2corollaryeq1}), note that the lower bound follows from $\mathbb{E}[W_{p}(\mu_{n},\mu)]\geq b_{n,p}(\mu)$ and the fact that by Theorem \ref{main1},
\[\left(\mathbb{E}[W_{p}(\mu_{n},\mu)^{p}]\right)^{1/p}\geq\mathbb{E}[W_{1}(\mu_{n},\mu)]\geq\frac{1}{22\sqrt{\ln(2n)}}\cdot\frac{1}{\sqrt{n}}\cdot b_{1,1}(\mu)>0.\]
To prove the upper bound in (\ref{main2corollaryeq1}), let
\[\alpha_{0}=\limsup_{n\to\infty}\frac{\ln\left(\mathbb{E}[W_{p}(\mu_{n},\mu)^{p}]\right)^{1/p}}{\ln n}.\]
Note that $\alpha_{0}\leq 0$ since $\mu\in\mathcal{P}_{p}(E)$. If $\alpha_{0}\leq-\frac{1}{2p}$ then the upper bound in (\ref{main2corollaryeq1}) is trivial. If not, then $0\geq\alpha_{0}>-\frac{1}{2p}$. Fix $\epsilon>0$ so that $\alpha_{0}-\epsilon>-\frac{1}{2p}$. Take $\beta=\epsilon-\alpha_{0}$. Then $0<\epsilon\leq\beta<\frac{1}{2p}$. By the definition of $\alpha_{0}$, there are infinitely many values of $n\in\mathbb{N}$ for which $\ln\left(\mathbb{E}[W_{p}(\mu_{n},\mu)^{p}]\right)^{1/p}\geq(\alpha_{0}-\epsilon)\ln n=-\beta\ln n$, and so
\[\limsup_{n\to\infty}n^{\beta}\cdot\left(\mathbb{E}[W_{p}(\mu_{n},\mu)^{p}]\right)^{1/p}\geq 1.\]
So by (\ref{main2corollaryeq2}), we have
\[\limsup_{n\to\infty}n^{\beta}\cdot b_{n,p}(\mu)\geq 1.\]
Thus, there are infinitely many values of $n\in\mathbb{N}$ for which $n^{\beta}\cdot b_{n,p}(\mu)\geq\frac{1}{2}$, so
\[\limsup_{n\to\infty}\frac{\ln b_{n,p}(\mu)}{\ln n}\geq-\beta=\alpha_{0}-\epsilon.\]
Taking $\epsilon\to 0$ completes the proof of (\ref{main2corollaryeq1}).
\end{proof}
\begin{proof}[Proof of Corollary \ref{main4}]
By Theorem \ref{main2},
\[\left(\mathbb{E}[W_{p}(\mu_{n},\mu)^{p}]\right)^{1/p}\leq5\cdot\sum_{k=0}^{\lfloor\log_{2}n\rfloor}\left(\frac{2^{k}}{n}\right)^{1/(2p)}\cdot b_{2^{k},p}(\mu),\]
and by Theorem \ref{main3},
\[b_{2^{k},p}(\mu)\leq\left(2\sum_{j=0}^{k}2^{(\frac{p}{q}-1)(k-j)}\cdot e_{2^{j},q}(\mu)^{p}\right)^{1/p}\leq 2\sum_{j=0}^{k}2^{(\frac{1}{q}-\frac{1}{p})(k-j)}\cdot e_{2^{j},q}(\mu),\]
for every $k\in\mathbb{N}\cup\{0\}$. Therefore,
\begin{eqnarray}\label{main4eq1}
\left(\mathbb{E}[W_{p}(\mu_{n},\mu)^{p}]\right)^{1/p}&\leq&
10\cdot\sum_{k=0}^{\lfloor\log_{2}n\rfloor}\left(\frac{2^{k}}{n}\right)^{1/(2p)}\cdot\sum_{j=0}^{k}2^{(\frac{1}{q}-\frac{1}{p})(k-j)}\cdot e_{2^{j},q}(\mu)\\&=&
10\cdot\sum_{j=0}^{\lfloor\log_{2}n\rfloor}n^{-1/(2p)}\cdot e_{2^{j},q}(\mu)\cdot\sum_{k=j}^{\lfloor\log_{2}n\rfloor} 2^{\frac{k}{2p}+(\frac{1}{q}-\frac{1}{p})(k-j)},\nonumber
\end{eqnarray}
where we switch the order of the two sums. The inner sum is a geometric series and we now bound this geometric series in three different cases.

Case 1: If $q>2p$, then
\[\sum_{k=j}^{\lfloor\log_{2}n\rfloor}2^{\frac{k}{2p}+(\frac{1}{q}-\frac{1}{p})(k-j)}\leq
\sum_{k=j}^{\infty}2^{\frac{k}{2p}+(\frac{1}{q}-\frac{1}{p})(k-j)}=\frac{2^{j/(2p)}}{1-2^{\frac{1}{q}-\frac{1}{2p}}}\leq\frac{2\cdot 2^{j/(2p)}}{\frac{1}{2p}-\frac{1}{q}},\]
where the last step follows from Lemma \ref{2a}. By (\ref{main4eq1}), it follows that
\[\left(\mathbb{E}[W_{p}(\mu_{n},\mu)^{p}]\right)^{1/p}\leq\frac{20}{\frac{1}{2p}-\frac{1}{q}}\sum_{j=0}^{\lfloor\log_{2}n\rfloor}\left(\frac{2^{j}}{n}\right)^{1/(2p)}\cdot e_{2^{j},q}(\mu).\]

Case 2: If $q=2p$, then
\[\sum_{k=j}^{\lfloor\log_{2}n\rfloor} 2^{\frac{k}{2p}+(\frac{1}{q}-\frac{1}{p})(k-j)}=\sum_{k=j}^{\lfloor\log_{2}n\rfloor}2^{j/(2p)}=(\lfloor\log_{2}n\rfloor-j+1)\cdot 2^{j/(2p)}\leq(\lfloor\log_{2}n\rfloor+1)\cdot 2^{j/(2p)}.\]
By (\ref{main4eq1}), it follows that
\[\left(\mathbb{E}[W_{p}(\mu_{n},\mu)^{p}]\right)^{1/p}\leq10(\lfloor\log_{2}n\rfloor+1)\cdot\sum_{j=0}^{\lfloor\log_{2}n\rfloor}\left(\frac{2^{j}}{n}\right)^{1/(2p)}\cdot e_{2^{j},q}(\mu).\]

Case 3: If $q<2p$, then
\[\sum_{k=j}^{\lfloor\log_{2}n\rfloor}2^{\frac{k}{2p}+(\frac{1}{q}-\frac{1}{p})(k-j)}\leq
\sum_{k=-\infty}^{\lfloor\log_{2}n\rfloor}2^{\frac{k}{2p}+(\frac{1}{q}-\frac{1}{p})(k-j)}=
\frac{2^{\frac{k_{0}}{2p}+(\frac{1}{q}-\frac{1}{p})(k_{0}-j)}}{1-2^{\frac{1}{2p}-\frac{1}{q}}},\]
where $k_{0}=\lfloor\log_{2}n\rfloor$. Since the exponent
\begin{eqnarray*}
\frac{k_{0}}{2p}+\left(\frac{1}{q}-\frac{1}{p}\right)(k_{0}-j)&=&\left(\frac{1}{q}-\frac{1}{2p}\right)k_{0}+\left(\frac{1}{p}-\frac{1}{q}\right)j\\&\leq&
\left(\frac{1}{q}-\frac{1}{2p}\right)\log_{2}n+\left(\frac{1}{p}-\frac{1}{q}\right)j,
\end{eqnarray*}
it follows that
\[\sum_{k=j}^{\lfloor\log_{2}n\rfloor}2^{\frac{k}{2p}+(\frac{1}{q}-\frac{1}{p})(k-j)}\leq\frac{1}{1-2^{\frac{1}{2p}-\frac{1}{q}}}\cdot n^{\frac{1}{q}-\frac{1}{2p}}\cdot 2^{(\frac{1}{p}-\frac{1}{q})j}\leq\frac{2}{\frac{1}{q}-\frac{1}{2p}}\cdot n^{\frac{1}{q}-\frac{1}{2p}}\cdot 2^{(\frac{1}{p}-\frac{1}{q})j},\]
where the last step uses Lemma \ref{2a}. So by (\ref{main4eq1}), we have
\begin{eqnarray*}
\left(\mathbb{E}[W_{p}(\mu_{n},\mu)^{p}]\right)^{1/p}&\leq&
\frac{20}{\frac{1}{q}-\frac{1}{2p}}\cdot\sum_{j=0}^{\lfloor\log_{2}n\rfloor}n^{-1/(2p)}\cdot e_{2^{j},q}(\mu)\cdot n^{\frac{1}{q}-\frac{1}{2p}}\cdot 2^{(\frac{1}{p}-\frac{1}{q})j}\\&=&
\frac{20}{\frac{1}{q}-\frac{1}{2p}}\cdot\sum_{j=0}^{\lfloor\log_{2}n\rfloor}\left(\frac{2^{j}}{n}\right)^{\frac{1}{p}-\frac{1}{q}}\cdot e_{2^{j},q}(\mu).
\end{eqnarray*}
\end{proof}
\begin{lemma}\label{BKlemma}
Let $n\in\mathbb{N}$ and $\frac{1}{\max(n,2)}\leq p\leq\frac{1}{2}$. If $W_{1},\ldots,W_{n}$ are independent random variables such that $\mathbb{P}(W_{i}=1)=p$ and $\mathbb{P}(W_{i}=0)=1-p$ for all $i$, then
\[\mathbb{E}\left|\frac{1}{n}\sum_{i=1}^{n}W_{i}-p\right|\geq\frac{1}{2}\sqrt{\frac{p}{n}}.\]
\end{lemma}
\begin{proof}
If $n=1$ then $p=\frac{1}{2}$ so $\mathbb{E}|W_{1}-p|=\mathbb{E}|W_{1}-\frac{1}{2}|=\frac{1}{2}$. If $n\geq 2$, then $\frac{1}{n}\leq p\leq1-\frac{1}{n}$ and so by \cite[Theorem 1]{BK}, we have
\[\mathbb{E}\left|\frac{1}{n}\sum_{i=1}^{n}W_{i}-p\right|\geq\frac{1}{\sqrt{2}}\cdot\frac{1}{n}\cdot\sqrt{np(1-p)}=\sqrt{\frac{p(1-p)}{2n}}\geq\sqrt{\frac{p}{4n}}.\]
\end{proof}
\begin{proof}[Proof of Corollary \ref{supmu}]
Taking $q\to\infty$ in Corollary \ref{main4} gives
\[\left(\mathbb{E}[W_{p}(\mu_{n},\mu)^{p}]\right)^{1/p}\leq
40p\cdot\sum_{k=0}^{\lfloor\log_{2}n\rfloor}\left(\frac{2^{k}}{n}\right)^{1/(2p)}\cdot\lim_{q\to\infty}e_{2^{k},q}(\mu).\]
For every $\epsilon>0$ such that $N(E,\epsilon)\leq 2^{k}$, there exists $S\subset E$ with cardinality at most $2^{k}$ such that $d_{E}(x,S)\leq\epsilon$ for all $x\in E$, and so by Lemma \ref{enpalternative}, we have $e_{2^{k},q}(\mu)\leq\epsilon$ for all $q\geq 1$. Thus, by the definition of $h(2^{k})$, we have $e_{2^{k},q}(\mu)\leq h(2^{k})$ for all $q\geq 1$. This proves the upper bound in Corollary \ref{supmu}.

To prove the lower bound, we need the notion of $\epsilon$-separating set. Let $\epsilon>0$. A subset $S\subset E$ is {\it $\epsilon$-separating} if $d_{E}(x,y)>\epsilon$ whenever $x\neq y$ in $S$.

Fix an integer $1\leq r\leq n$. Our goal is to show that
\begin{equation}\label{supmugoal}
\sup_{\mu\in\mathcal{P}_{p}(E)}\left(\mathbb{E}[W_{p}(\mu_{n},\mu)^{p}]\right)^{1/p}\geq\frac{1}{4}\cdot\left(\frac{r}{n}\right)^{1/(2p)}h(r).
\end{equation}
Let $0<\delta<h(r)$ and let $\epsilon=h(r)-\delta$.

{\bf Claim:} There exists an $\epsilon$-separating subset $S$ of $E$ with cardinality $|S|=\max(r,2)$.

To prove the claim, let $S_{0}$ be a maximal $\epsilon$-separating subset of $E$. Then by maximality, $S_{0}$ is an $\epsilon$-cover of $E$, i.e., every $y\in E$ has distance at most $\epsilon$ from some point in $S_{0}$. So $N(E,\epsilon)\leq|S_{0}|$. But by the definition of $h(r)$, we have $N(E,\epsilon)=N(E,h(r)-\delta)>r$. Therefore, $|S_{0}|>r$. The claim follows by taking $S$ to be any subset of $S_{0}$ with cardinality $|S|=\max(r,2)$. This proves the claim.

Take
\[\mu=\frac{1}{|S|}\sum_{x\in S}\delta_{x}.\]
Since $S$ is $\epsilon$-separating, if $\nu$ is another probability measure on $S$, then
\begin{equation}\label{supmueq1}
W_{p}(\mu,\nu)^{p}\geq\frac{\epsilon^{p}}{2}\sum_{x\in S}|\mu(\{x\})-\nu(\{x\})|.
\end{equation}
This is because $d_{E}(x,y)\geq\epsilon\cdot I(x\neq y)$ for all $x,y\in S$, where $I(x\neq y)=\begin{cases}1&x\neq y\\0,&x=y\end{cases}$, and because the $p$-Wasserstein distance between any two probability measures with respect to the metric $I(x\neq y)$ coincides with the $p$th root of the total variation distance between the two measures. More precisely,
\begin{eqnarray*}
W_{p}(\mu,\nu)^{p}&=&\inf_{\gamma}\int_{S\times S}d_{E}(x,y)^{p}\,d\gamma(x,y)\\&\geq&
\inf_{\gamma}\epsilon^{p}\int_{S\times S}I(x\neq y)\,d\gamma(x,y)=\inf_{\gamma}\epsilon^{p}\left(1-\sum_{x\in S}\gamma(\{(x,x)\})\right),
\end{eqnarray*}
where the infima are over all probability measures $\gamma$ on $S\times S$ with marginal distributions $\mu$ and $\nu$. For every such measure $\gamma$, we have
\begin{eqnarray*}
1-\sum_{x\in S}\gamma(\{(x,x)\})&\geq&
1-\sum_{x\in S}\min(\mu(\{x\}),\nu(\{x\}))\\&=&
\frac{1}{2}\sum_{x\in S}\left[\mu(\{x\})+\nu(\{x\})-2\min(\mu(\{x\}),\nu(\{x\}))\right]\\&=&
\frac{1}{2}\sum_{x\in S}|\mu(\{x\})-\nu(\{x\})|.
\end{eqnarray*}
This proves (\ref{supmueq1}). Take $\nu=\mu_{n}$. Since $|S|=\max(r,2)$, we have $\frac{1}{\max(n,2)}\leq\frac{1}{|S|}\leq\frac{1}{2}$ so by Lemma \ref{BKlemma},
\[\mathbb{E}|\mu(\{x\})-\mu_{n}(\{x\})|\geq\frac{1}{2}\sqrt{\frac{\mu(\{x\})}{n}}=\frac{1}{2\sqrt{n|S|}},\]
for each $x\in S$. So by (\ref{supmueq1}) that
\[\mathbb{E}[W_{p}(\mu_{n},\mu)^{p}]\geq\frac{\epsilon^{p}|S|}{2}\cdot\frac{1}{2\sqrt{n|S|}}=\frac{\epsilon^{p}}{4}\sqrt{\frac{|S|}{n}}\geq\frac{\epsilon^{p}}{4}\sqrt{\frac{r}{n}}=\frac{(h(r)-\delta)^{p}}{4}\sqrt{\frac{r}{n}}.\]
Thus,
\[\sup_{\mu\in\mathcal{P}_{p}(E)}\mathbb{E}[W_{p}(\mu_{n},\mu)^{p}]\geq\frac{(h(r)-\delta)^{p}}{4}\sqrt{\frac{r}{n}}.\]
Taking $\delta\to0$, we obtain (\ref{supmugoal}) and the result follows.
\end{proof}
\begin{proof}[Proof of Corollary \ref{main4beta}]
Before we apply Corollary \ref{main4}, let us show that $p<q<2p$. Since $\frac{1}{p+\epsilon}-\beta\leq\frac{1}{p+\epsilon}<\frac{1}{p}$, we have $q>p$. Since $\beta<\frac{1}{2p}$, we have $\frac{3}{4p}+\frac{\beta}{2}<\frac{1}{p}$ and so
\[\epsilon<\frac{p^{2}}{2}\left(\frac{1}{2p}-\beta\right)=\frac{1}{2}\cdot\frac{\frac{1}{2}-\beta p}{1/p}<\frac{1}{2}\cdot\frac{\frac{1}{2}-\beta p}{\frac{3}{4p}+\frac{\beta}{2}},\]
which implies
\[p+\epsilon<\frac{1}{2}\left(2p+\frac{\frac{1}{2}-\beta p}{\frac{3}{4p}+\frac{\beta}{2}}\right)=\frac{1}{\frac{3}{4p}+\frac{\beta}{2}}.\]
Hence,
\begin{equation}\label{main4betaeq1}
\frac{1}{p+\epsilon}-\beta>\frac{3}{4p}-\frac{\beta}{2}>\frac{3}{4p}-\frac{1}{4p}=\frac{1}{2p},
\end{equation}
and so $q<2p$. Thus, we have proved $p<q<2p$. By Corollary \ref{main4},
\[\left(\mathbb{E}[W_{p}(\mu_{n},\mu)^{p}]\right)^{1/p}\leq\frac{20}{\frac{1}{q}-\frac{1}{2p}}\cdot\sum_{k=0}^{\lfloor\log_{2}n\rfloor}\left(\frac{2^{k}}{n}\right)^{\frac{1}{p}-\frac{1}{q}}\cdot e_{2^{k},q}(\mu).\]
Since by (\ref{main4betaeq1}), $\frac{1}{p+\epsilon}-\beta>\frac{3}{4p}-\frac{\beta}{2}$, the denominator
\[\frac{1}{q}-\frac{1}{2p}=\left(\frac{1}{p+\epsilon}-\beta\right)-\frac{1}{2p}>\left(\frac{3}{4p}-\frac{\beta}{2}\right)-\frac{1}{2p}=\frac{1}{4p}-\frac{\beta}{2}=\frac{1}{2}\left(\frac{1}{2p}-\beta\right),\]
so
\begin{equation}\label{main4betaeq2}
\left(\mathbb{E}[W_{p}(\mu_{n},\mu)^{p}]\right)^{1/p}\leq\frac{40}{\frac{1}{2p}-\beta}\cdot\sum_{k=0}^{\lfloor\log_{2}n\rfloor}\left(\frac{2^{k}}{n}\right)^{\frac{1}{p}-\frac{1}{q}}\cdot e_{2^{k},q}(\mu).
\end{equation}
Observe that
\begin{equation}\label{main4betaeq3}
\frac{1}{p}-\frac{1}{q}-\beta=\frac{1}{p}-\frac{1}{p+\epsilon}=\frac{\epsilon}{p(p+\epsilon)}>0.
\end{equation}
Let $a=\limsup_{n\to\infty}n^{\beta}\cdot e_{n,q}(\mu)$. Fix $\delta>0$. Then there exists $N\in\mathbb{N}$ such that $e_{n,q}(\mu)\leq(a+\delta)n^{-\beta}$ for all $n\geq N$. So for all $n\geq 2N$,
\begin{eqnarray*}
\sum_{k=\lfloor\log_{2}N\rfloor+1}^{\lfloor\log_{2}n\rfloor}\left(\frac{2^{k}}{n}\right)^{\frac{1}{p}-\frac{1}{q}}\cdot e_{2^{k},q}(\mu)&\leq&
(a+\delta)\sum_{k=\lfloor\log_{2}N\rfloor+1}^{\lfloor\log_{2}n\rfloor}\left(\frac{2^{k}}{n}\right)^{\frac{1}{p}-\frac{1}{q}}\cdot(2^{k})^{-\beta}\\&\leq&
(a+\delta)n^{\frac{1}{q}-\frac{1}{p}}\cdot\sum_{k=-\infty}^{\lfloor\log_{2}n\rfloor}2^{k(\frac{1}{p}-\frac{1}{q}-\beta)}\\&=&
(a+\delta)n^{\frac{1}{q}-\frac{1}{p}}\cdot 2^{\lfloor\log_{2}n\rfloor(\frac{1}{p}-\frac{1}{q}-\beta)}/(1-2^{-(\frac{1}{p}-\frac{1}{q}-\beta)})\\&\leq&
(a+\delta)n^{-\beta}/(1-2^{-(\frac{1}{p}-\frac{1}{q}-\beta)})\\&\leq&
2(a+\delta)n^{-\beta}/\left(\frac{1}{p}-\frac{1}{q}-\beta\right)
\\&=&
2(a+\delta)n^{-\beta}\cdot\frac{p(p+\epsilon)}{\epsilon},
\end{eqnarray*}
where the second step follows by extending the sum to start from $k=-\infty$, the third step follows from the fact that it is a geometric series which converges by (\ref{main4betaeq3}), the fourth step follows by using $\lfloor\log_{2}n\rfloor\leq\log_{2}n$, the fifth step follows from Lemma \ref{2a} and the last step follows from (\ref{main4betaeq3}). So by (\ref{main4betaeq2}),
\begin{eqnarray*}
\left(\mathbb{E}[W_{p}(\mu_{n},\mu)^{p}]\right)^{1/p}&\leq&
\frac{40}{\frac{1}{2p}-\beta}\left(\sum_{k=0}^{\lfloor\log_{2}N\rfloor}\left(\frac{2^{k}}{n}\right)^{\frac{1}{p}-\frac{1}{q}}\cdot e_{2^{k},q}(\mu)+\sum_{k=\lfloor\log_{2}N\rfloor+1}^{\lfloor\log_{2}n\rfloor}\left(\frac{2^{k}}{n}\right)^{\frac{1}{p}-\frac{1}{q}}\cdot e_{2^{k},q}(\mu)\right)
\\&\leq&\frac{40}{\frac{1}{2p}-\beta}\left(\sum_{k=0}^{\lfloor\log_{2}N\rfloor}\left(\frac{2^{k}}{n}\right)^{\frac{1}{p}-\frac{1}{q}}\cdot e_{2^{k},q}(\mu)+2(a+\delta)n^{-\beta}\cdot\frac{p(p+\epsilon)}{\epsilon}\right).
\end{eqnarray*}
Since $\beta-(\frac{1}{p}-\frac{1}{q})<0$ by (\ref{main4betaeq3}), multiplying by $n^{\beta}$ on both sides and taking $\limsup_{n\to\infty}$, we obtain
\[\limsup_{n\to\infty}n^{\beta}\cdot\left(\mathbb{E}[W_{p}(\mu_{n},\mu)^{p}]\right)^{1/p}\leq
\frac{40}{\frac{1}{2p}-\beta}\left(2(a+\delta)\cdot\frac{p(p+\epsilon)}{\epsilon}\right).\]
Taking $\delta\to 0$ and using the fact that $\epsilon<\frac{p^{2}}{2}(\frac{1}{2p}-\beta)\leq\frac{p^{2}}{2}\cdot\frac{1}{2p}=\frac{p}{4}$, we get
\[\limsup_{n\to\infty}n^{\beta}\cdot\left(\mathbb{E}[W_{p}(\mu_{n},\mu)^{p}]\right)^{1/p}\leq
\frac{80a}{\frac{1}{2p}-\beta}\cdot\frac{p(p+\frac{p}{4})}{\epsilon}=\frac{100p^{2}a}{(\frac{1}{2p}-\beta)\epsilon}.\]
\end{proof}

\end{document}